\documentclass[10pt]{article}

\usepackage{amsmath,amsthm,amssymb}
\usepackage{caption,graphicx,subfig}
\usepackage[text={6.5in,9.25in},centering]{geometry}
\usepackage{paralist}
\usepackage[sort&compress,numbers]{natbib}
\usepackage{cleveref}
\usepackage{xcolor}
\usepackage[normalem]{ulem}

\setcaptionmargin{0.25in}

\makeatletter\@addtoreset{equation}{section}\makeatother

\newtheorem{thm}{Theorem}[section] 
\newtheorem{lem}[thm]{Lemma}  
\newtheorem{cor}[thm]{Corollary} 
\newtheorem{prop}[thm]{Proposition}

\theoremstyle{definition}

\newcommand{\drm}{\mathrm{d}}
\newcommand{\R}{\mathbb{R}}

\newcommand{\cL}{\mathcal{L}}

\newcommand{\be}{\begin{equation}}
\newcommand{\ee}{\end{equation}}

\newcommand{\bigO}{\mathcal{O}}

\begin{document}

\title{The speed of traveling waves in a FKPP-Burgers system}

\author{
Jason J. Bramburger\thanks{Department of Applied Mathematics, University of Washington, Seattle, WA, 98105}
\and
Christopher Henderson\thanks{Department of Mathematics, University of Arizona, Tuscon, AZ, 85721}
}

\date{}
\maketitle

\begin{abstract}
	We consider a coupled reaction-advection-diffusion system based on the Fisher-KPP and Burgers equations.  These equations serve as a one-dimensional version of a model for a reacting fluid in which the arising density differences induce a buoyancy force advecting the fluid.  We study front propagation in this system through the lens of traveling waves solutions.  We are able to show two quite different behaviors depending on whether the coupling constant $\rho$ is large or small.  First, it is proved that there is a threshold $\rho_0$ under which the advection has no effect on the speed of traveling waves (although the advection is nonzero). Second, when $\rho$ is large, wave speeds must be at least $\mathcal{O}(\rho^{1/3})$.  These results together give that there is a transition from pulled to pushed waves as $\rho$ increases.  Because of the complex dynamics involved in this and similar models, this is one of the first precise results in the literature on the effect of the coupling on the traveling wave solution.  We use a mix of ordinary and partial differential equation methods in our analytical treatment, and we supplement this with a numerical treatment featuring newly created methods to understand the behavior of the wave speeds.  Finally, various conjectures and open problems are formulated.
\end{abstract}


\section{Introduction} 

The existence and stability of traveling wave solutions to reaction-diffusion equations of the form
\be\label{e.general_rd}
	T_t + \nabla\cdot\left(u T\right) = \Delta T + f(T)
\ee
has been an area of great activity during the past century~\cite{Fisher,KPP,Xin_book}.  Of particular interest in recent years is the influence of the advection $u$ on the speed of fronts~\cite{BHN,ConstantinKiselevObermanRyzhik, ElSmaily, ElSmailyKirsch, HamelZlatos, KiselevRyzhik, MajdaSouganidis, NolenXin1, NolenXin2, NolenXin3, NolenXin4, NovikovRyzhik, RyzhikZlatos, Zlatos, Zlatos2, Zlatos3}.  Generally one finds that unbiased advection ``stirs up'' a reaction, increasing the speed of traveling waves~\cite{BHN}.  The mechanism for this is complex but intimately connected to issues in homogenization and mixing problems.

However, most previous works focus on settings where the advection is imposed from the outside, that is, $u$ is independent of the evolution of $T$. Often works consider $u$ that is time independent and has rigid structure, such as periodicity or stationary ergodicity.  For many physical systems, the reacting quantity and the advection influence each other, evolving in time together. These coupled systems, which can be thought of as ``reacting'' active scalars~\cite{Kiselev}, are much more difficult to analyze as standard techniques that strongly use the structure of $u$ do not apply.  Two prominent examples are the dynamics of a population of chemotactic bacteria~\cite{KellerSegel} and of the temperature of a fluid undergoing a chemical reaction (flame propagation)~\cite{MalhamXin}.  In this paper, we are interested in the latter setting and our particular focus is on how the strength of the coupling between $u$ and $T$ affects the speed of traveling waves.

The model that we consider is a one-dimensional reactive Burgers equation that serves as a simplified model for the situation in which a fluid undergoes a chemical reaction that increases its temperature and the density differences that arise from the temperature being out of equilibrium induce a buoyancy (gravity) force (cf.~the discussion of~\eqref{e.Boussinesq} below).  Precisely, the equations are given by
\be\label{FKPP_Burgers}
	\begin{split}
		T_t - T_{xx} + (uT)_x &= T(1-T),\\
		u_t - \nu u_{xx} + uu_x &= \rho T(1-T).
	\end{split}
\ee
Here $T(x,t)$ and $u(x,t)$ denote the temperature and velocity of the fluid, respectively, each dependent on space $x\in\R$ and time $t \geq 0$. The quantity $\nu \geq 0$ represents the viscosity of the fluid. The first equation in~\eqref{FKPP_Burgers} is the Fisher-KPP (FKPP) equation~\cite{Fisher,KPP}, which is a prototypical equation used in population dynamics and combustion.  In this setting we use the latter interpretation.  The temperature, $T$,  increases due to combustion and is advected by the solution of the second equation, $u$, which is the Burgers equation with a Boussinesq-type gravity term on the right hand side, given here by $\rho T(1-T)$.  The parameter $\rho$ is a positive constant quantifying the strength of coupling.

The goal of this paper is to prove the existence of traveling wave solutions to~\eqref{FKPP_Burgers} and understand how their behavior depends on the parameters $\rho$ and $\nu$. That is, abusing notation, we seek solutions of the form $T(x,t) = T(x-ct)$ and $u(x,t) = U(x-ct)$ where the profiles satisfy the limits 
\begin{equation}\label{TWLimits}
	(T(-\infty),U(-\infty)) = (1,u_0), \quad (T(\infty),U(\infty)) = (0,0),		
\end{equation}
for some $u_0 \in \R$, where $(1,u_0)$ is a spatially-independent equilibrium solution of~\eqref{FKPP_Burgers}. Throughout, we adopt the nomenclature from the usual Burgers equation that \eqref{FKPP_Burgers} with $\nu = 0$ is {\em inviscid}, while the equation having $\nu > 0$ is {\em viscous}. In particular, our interest is in the relationship between the range of speeds $c$ that traveling wave solutions move at and the coupling constant $\rho$. Our main results show that there is a transition in behavior between small values of $\rho$ where the minimal admissible speed of any traveling wave is $2$ (as in the FKPP case, where $u\equiv 0$) and large values of $\rho$ where the advection ``speeds up'' the minimal speed to be $O(\rho^{1/3})$. This corresponds to a transition between ``pulled'' traveling waves, typical for FKPP, and waves that are ``pushed'' by the advection $u$.  

The meaning of equation~\eqref{FKPP_Burgers} can be understood with reference to two similar models for the same physical setting.  First, it is a one-dimensional version of the class of reactive-Boussinesq models in which $T$ is governed by~\eqref{e.general_rd}, while $u$ satisfies $\nabla \cdot u = 0$ and
\be\label{e.Boussinesq}
	\cL_{\rm fluid}(u) = \rho T \hat z,
\ee
where $\hat z$ is the upward pointing unit vector and $\cL_{\rm fluid}$ is the any operator for a fluid equation. For example, one might have $\cL_\mathrm{fluid}$ being the Navier-Stokes equations 
\be
	\cL_{\rm fluid}(u)
		= u_t + u \cdot \nabla u - \Delta u + \nabla p,
\ee
where $p$ is the pressure term. In this setting, the $\rho T$ is a gravity term and $\rho$ is the Rayleigh number, which is a nondimensional constant roughly representing the strength of gravity with respect to the internal properties of the fluid. The dynamics of solutions of~\eqref{e.general_rd} coupled with \eqref{e.Boussinesq} are known to be complex.  To our knowledge the model was introduced in~\cite{MalhamXin}, with traveling waves being constructed in various settings~\cite{BerestyckiConstantinRyzhik, ConstantinLewickaRyzhik, Lewicka, LewickaMucha, Henderson_Boussinesq, TexierPicardVolpert}, and some general bounds for the Cauchy problem were obtained in~\cite{ConstantinKiselevRyzhik}.  In addition, it is known that there is a destabilizing bifurcation that depends on the parameter $\rho$ so that planar waves (in which $u\equiv 0$) are stable for $\rho$ small but unstable for $\rho$ large~\cite{ConstantinKiselevRyzhik, TexierPicardVolpert, TexierPicardVolpert_bifurcation, VladimirovaRosner}. In these works, however, little is known about the case $\rho\gg 1$; in particular, the affect of $\rho$ on the speed of waves is unknown. Furthermore, to our knowledge, there is no conjectured asymptotics of the speed $c=c(\rho)$ as $\rho\to\infty$.

Closer to our setting is the system considered by Constantin et al.~\cite{CRRV}, given by
\be\label{e.CRRV}
	\begin{split}
		T_t - T_{xx} + u T_x &= f(T),\\
		u_t - uu_x &= \rho T.
	\end{split}
\ee
The authors construct traveling waves and investigate quenching with a particular focus on the structure of solutions. They prove that there are $0 < \rho_1 < \rho_2$ such that the following takes place. When $\rho \in (0,\rho_1]$ solutions decompose into a traveling wave moving to the right and an accelerated shock wave moving leftwards. When $\rho \geq \rho_2$ these traveling waves additionally have a wave fan component. They also show that quenching occurs on any length scale when $\rho$ is sufficiently large. We note that quenching is not possible in our context as we use the choice $f(T) = T(1-T)$ instead of an ignition type nonlinearity as in~\cite{CRRV}. In the present work, we are most interested in the behavior of the traveling wave speed as $\rho$ varies while, in~\cite{CRRV}, the authors are focused instead on the structure of the waves and the phenomenon of quenching; as such, they do not address how the speed of traveling waves is affected by the value of $\rho$.

The two main differences between~\eqref{FKPP_Burgers} and~\eqref{e.CRRV} that we point out are (i) the advective term: $(uT)_x$ in~\eqref{FKPP_Burgers} versus $uT_x$ in~\eqref{e.CRRV}, and (ii) the gravity term: $\rho T(1-T)$ in~\eqref{FKPP_Burgers} versus $\rho T$ in~\eqref{e.CRRV}. The reason for our choice of advective term is that it allows the total internal energy, in the absence of reaction, to be conserved. Explicitly, if $\tilde T(x,t)$ satisfies
\[
	\tilde T_t - \tilde T_{xx} + (u \tilde T)_x = 0,
\]
with initial condition $\tilde T(x,0) \in L^1(\R)$, then $\int \tilde T(x,t) \drm x$ is independent of $t$.  This would not be true if $(u \tilde T)_x$ were replaced by $u \tilde T_x$. (We note that the reactive-Boussinesq equation~\eqref{e.Boussinesq} also conserves energy in the absence of reaction due to the assumption that $\nabla\cdot u = 0$.)  Hence, in~\eqref{FKPP_Burgers} any change to the total internal energy of the system arises through the reaction. Turning to the the gravity term, there are two reasons for our choice of $\rho T(1-T)$. First, in~\cite{CRRV}, it is observed that with the choice $\rho T$, the advection is unbounded, which is not physical.  Second, with this choice, if $T$ is in equilibrium, no advection is induced. This is consistent with the reactive-Boussinesq model~\eqref{e.Boussinesq}.

Our proofs of the existence of traveling wave solutions to \eqref{FKPP_Burgers} rely on converting the problem to finding heteroclinic connections to an associated spatial ordinary differential equation (ODE) in the traveling wave independent variable $\xi = x-ct$. The analysis naturally splits into distinct cases for the inviscid and the viscous equation, with the former leading to a planar ODE and the latter leading to an ODE in three dimensions. This means that although the approach to the inviscid and viscous equations are similar, different arguments are required to prove the existence of traveling wave solutions.  

In the inviscid case, we show that traveling waves, when they exist, must be monotone; however, due to the increased complexity of three dimensional ODE systems, we are unable to establish this in the viscous setting.  In this work, we focus exclusively on traveling waves with monotone profiles $(T,U)$. This decision is not arbitrary, as non-monotone solutions to reaction-diffusion systems with monotone dynamics have been proven to be unstable in similar systems~\cite{Hagan,Volpert}. We note that our system differs from these previous studies since they are restricted to scalar equations or linear advection, whereas system~\eqref{FKPP_Burgers} has two components and the advection is nonlinear. In addition, we present numerical evidence below that the traveling waves that are stable are monotone (see \Cref{fig:Waves,fig:Waves2}). While we believe that non-monotone waves (if they exist) are similarly unstable in~\eqref{FKPP_Burgers}, we leave a proof of this to a follow-up paper that hopes to also detail the stability properties of the monotone fronts proven to exist in this work.

The remainder of this paper is organized as follows. In Section~\ref{sec:Results} we present our main results, split up between the inviscid and viscous cases, along with our numerical findings which complement our analysis. Particularly, we provide results from numerical simulations to observe the transition from pulled to pushed traveling waves and demonstrate how recent numerical bounding techniques due to \cite{BramburgerGoluskin} can be employed to determine the asymptotics of the minimum wave speed in the large $\rho$ parameter regime. Details on this numerical bounding technique can be found in the appendix. The proofs of the main results are given in Sections~\ref{sec:Inviscid} and~\ref{sec:Viscous}, with the former featuring the results in the inviscid case and the latter focussing on the viscous case. The paper then concludes with Section~\ref{sec:Discussion} where a brief discussion of our findings and potential avenues for future work can be found.

\subsubsection*{Acknowledgements}

CH was partially supported by NSF grant DMS-2003110.


\section{Main Results}\label{sec:Results} 

As discussed in the introduction, our investigation of~\eqref{FKPP_Burgers} naturally splits into the inviscid ($\nu = 0$) and viscous ($\nu > 0$) cases. The results are slightly different and therefore broken up in what follows. Particularly, we obtain a more complete understanding of the inviscid case, primarily due to the dimensionality of the resulting spatial ODE used to find traveling wave solutions. The results for the viscous case therefore aim to demonstrate that many of the inviscid case results extend into $\nu > 0$, at least when $\nu$ is sufficiently small.  After stating our main theorems for both cases of \eqref{FKPP_Burgers} we then present our numerical findings that complement our analysis. This section then concludes with a brief discussion of related work which frames our findings within the larger body of research into traveling wave solutions to reaction-diffusion equations of the type \eqref{e.general_rd}.


\subsection*{The inviscid case} 

Our first result establishes the existence and some qualitative properties of traveling wave solutions of \eqref{FKPP_Burgers} in the inviscid case.

\begin{thm}\label{thm:Existence1} 
	Fix $\nu = 0$ in~\eqref{FKPP_Burgers}. There exists $c_*(\rho)$, finite for all $\rho > 0$, such that the system \eqref{FKPP_Burgers} has a unique nonnegative traveling wave solution $(T(x,t),u(x,t)) = (T(x-ct),U(x-ct))$ satisfying the limits \eqref{TWLimits} if and only if $c \geq c_*(\rho)$ and $u_0 = c+\rho -\sqrt{c^2 +\rho^2}$. Furthermore, $T$ and $U$ are monotone decreasing and satisfy the ordering
	\begin{equation}\label{Ordering}
		0 < U(x-ct) < \rho T(x-ct) < \rho
	\end{equation}
	for all $x\in\R$ and $t \geq 0$.
\end{thm}

Roughly, this yields that, similar to the FKPP equation, all traveling waves are monotonic and the set of speeds $c$ for which there is a corresponding traveling wave solution is
a closed half-line $[c_*(\rho),\infty)$ with minimal speed $c_*(\rho)$.  We extend these results with the following characterization of the minimum speed.

\begin{thm}\label{thm:MinSpeed1} 
	The minimal speed $c_*(\rho)$ in Theorem~\ref{thm:Existence1} satisfies the following: 
	\begin{compactenum}[(1)]
		\item $c_*(\rho) \geq 2$ for all $\rho > 0$.
		\item There exists a $\rho_0 \in [1,\frac{16}{3}]$ such that for all $\rho \in (0,\rho_0]$ we have $c_*(\rho) = 2$. 
		\item $\displaystyle
			\left(\frac{3}{2}\right)^{1/3} \leq\liminf_{\rho\to\infty} \frac{c_*(\rho)}{\rho^{1/3}}
				\leq \limsup_{\rho\to\infty} \frac{c_*(\rho)}{\rho^{1/3}}
				\leq \sqrt 3.$	
	\end{compactenum}
\end{thm}

We note that the $\mathcal{O}(\rho^{1/3})$ scaling obtained in \Cref{thm:MinSpeed1}(3) is not obvious.  In fact, a na\"ive argument might go as follows.  Since $U \approx u_0$ on the left, we might expect
\be
	c = 2 + u_0.
\ee
Indeed, if $U \equiv u_0$ then this equality would hold. Substituting in for $u_0$, considering $\rho \gg 1$, and using a Taylor expansion, yields
\be
	c = 2 + c - \frac{c^2}{2\rho} + \mathcal{O}\bigg(\frac{c^4}{\rho^3}\bigg).
\ee
From the leading order terms we conclude that $c \approx 2 \rho^{1/2}$, which, in view of \Cref{thm:MinSpeed1}, is not accurate at all.

The proofs of Theorem~\ref{thm:Existence1} and \ref{thm:MinSpeed1} are left to Section~\ref{sec:Inviscid} where we show that the problem of finding traveling wave solutions to \eqref{FKPP_Burgers} is equivalent to determining the existence of heteroclinic orbits in a planar ODE. The main step in the construction of traveling waves and in obtaining upper bounds on the minimum wave speed is in finding ``trapping regions'' of the ODE system.  On the other hand, the lower bound on the minimum wave speed (when $\rho \gg 1$) is obtained through the analysis of several integral quantities.  These arguments are reminiscent of partial differential equations (PDE) arguments based on analyzing the ``bulk-burning rate'' developed in~\cite{ConstantinKiselevObermanRyzhik}, though we are able to obtain more precise bounds in the present setting due to the particular structure of~\eqref{FKPP_Burgers}.

From a PDE point of view, one of the major difficulties associated to~\eqref{FKPP_Burgers} is the lack of a general comparison principle.   As a result, standard techniques based on the construction of sub- and super-solutions do not apply.  In particular, one cannot use $(1,u_0)$ as a super-solution to immediately obtain {\em a priori} bounds on $(T,U)$.  While it might be possible to construct more complicated sub- and super-solutions, doing so would require a precise understanding of the profiles of $T$ and $U$, which appears very difficult.  On the other hand, the construction of trapping regions has the advantage that one need only understand the relationship of $T$ and $U$ to each other.  Stated another way, to construct trapping regions, we need only understand the phase plane on the boundary of a set as opposed to on the entirety of the set.  Hence, we are able to obtain bounds via trapping regions despite knowing very little about the shape of $T$ and $U$.

As we discuss below, numerically we see that $c_*(\rho)$ appears to be monotonic in $\rho$, indicating that there exists a critical value $\rho_c>0$ such that $c_*(\rho) = 2$ if $\rho < \rho_c$ and $c_*(\rho) > 2$ if $\rho > \rho_c$. We are unable to establish this here and leave it as a conjecture for future investigation. Furthermore, the upper bound of $16/3$ on $\rho_0$ comes from Lemma~\ref{lem:LargeRho2} below where we give a lower bound on the minimum wave speed by ruling out the existence of said heteroclinic orbits. This lemma not only provides the upper bound on $\rho_0$, but shows that $c_*(\rho) \gtrsim \rho^{1/3}$, partially contributing to the result Theorem~\ref{thm:MinSpeed1}(3). Interestingly, the numerical findings below have led us to believe that the lower bound provided by Lemma~\ref{lem:LargeRho2} is asymptotically sharp in $\rho$. Attempts to prove this analytically remain elusive and would require one to bring down the upper bound on $c_*(\rho)$ provided in Lemma~\ref{lem:LargeRho}. The reader is directed to our discussion of the numerical findings below for more details on this.

\subsection*{The viscous case}

When $\nu>0$, many of the results above can be recovered, albeit with more difficulty.  In particular, we obtain the existence of traveling waves for a half-infinite set of speeds, although possibly not a line. We show that, for large $\rho$, all waves must have speed at least $\mathcal{O}(\rho^{1/3})$, though we are only able to construct waves for speeds like $\mathcal{O}(\rho^{1/2})$.  Finally, we establish a range of $\nu$ and $\rho$ for which the set of admissible speeds is $[2,\infty)$. In the small $\nu > 0$ parameter regime, our results indicate a regular perturbation of the inviscid case, although regularity of the wave speed has not been confirmed here. We present these results in the following theorem whose proof is left to Section~\ref{sec:Viscous}. 

\begin{thm}\label{thm:Viscous} 
	For every $\nu > 0$, there exists $\underline c_*(\nu,\rho)$ and $\overline c_*(\nu,\rho)$, positive and finite for all $\rho > 0$, such that the system \eqref{FKPP_Burgers} has a unique nonnegative monotone traveling wave solution $(T(x,t),u(x,t))=(T(x-ct),U(x-ct))$ satisfying the limits \eqref{TWLimits} if $c \geq \overline c_*(\nu,\rho)$ and only if $c \geq \underline c_*(\nu,\rho)$, with $u_0 = c + \rho - \sqrt{c^2 + \rho^2}$.
	
	The functions $\overline c_*(\nu,\rho)$, $\underline c_*(\nu,\rho)$ satisfy the following:
	\begin{compactenum}[(1)]
		\item For all $\nu,\rho > 0$ we have $2 \leq \underline c_*(\nu,\rho) \leq \overline c_*(\nu,\rho) < \infty$.
		\item For all $\nu > 0$ we have $\overline c_*(\nu,\rho) \to 2$ as $\rho \to 0^+$, and consequently $\underline c_*(\nu,\rho) \to 2$ as well.  
		\item There exists $\nu_* \geq 2$ so that for all $\nu \in (0,\nu_*)$ there exists $\rho_\nu \in (0,8]$ such that $\overline c_*(\nu,\rho) = \underline c_*(\nu,\rho) = 2$ for all $\rho \in (0,\rho_\nu]$. 
		\item For all $\nu > 0$ we have
			$\displaystyle \liminf_{\rho\to\infty} \frac{\underline c_*(\nu,\rho)}{\rho^{1/3}} \geq 1$ and $\displaystyle \limsup_{\rho\to\infty} \frac{\overline c_*(\nu,\rho)}{\rho^{1/2}}
				\leq 2$.
	\end{compactenum}
\end{thm}

Despite the results for the inviscid and viscous cases being similar, we note that the analysis is fundamentally different. That is, the spatial ODEs used to prove the existence of traveling waves have different dimensions and therefore require different analytical treatments. For this reason our results in the viscous setting are weaker than in the inviscid case. 
The most notable differences are the upper bound on $\overline c_*(\nu,\rho)$ being $\mathcal{O}(\rho^{1/2})$ and the fact that we can only establish $\underline c_*(\nu,\rho) = \overline c_*(\nu,\rho)$ when $\nu$ and $\rho$ are small, so that we do not know if the set of admissible speeds make up a half-line. The former issue is technical and discussed in greater detail in \Cref{sec:Viscous}. We note that the latter issue is related to the fact that two initially ordered curves must cross in 2D in order to ``flip'' their order, while in 3D they can ``wind'' around each other without crossing (see \Cref{lem:MinSpeed} for this crossing argument in 2D).

Before discussing our numerical findings, we note the following subtlety in the viscous case. In addition to the steady-state $(1, c + \rho - \sqrt{c^2 + \rho^2})$ which forms the backwards limit of traveling waves in Theorem~\ref{thm:Viscous}, there is a second steady state $(1, c + \rho + \sqrt{c^2 + \rho^2})$ which could also serve as an asymptotic rest state for traveling wave solutions of \eqref{FKPP_Burgers}.  Due to the particular dynamics presented in Section~\ref{sec:Inviscid}, traveling waves connecting this second steady state to $(0,0)$ can be proven to not exist in the inviscid setting (see \Cref{lem:big_equilibrium}); however, it is not clear that these can be ruled out in the viscous case.  Analytically, it is clear that, were these to exist, they would not be monotonic and, as discussed in the introduction, non-monotone solutions to reaction-diffusion systems with monotone dynamics have been shown to be unstable in similar systems \cite{Hagan,Volpert}. In addition, the numerics (discussed below) seem to indicate the stability of the monotone wave connecting $(1, c + \rho - \sqrt{c^2 + \rho^2})$ and $(0,0)$.  For these reasons, we restrict our attention to these monotone waves in this work. As an aside, it is worth noting that the results on the instability of non-monotone traveling waves in~\cite{Hagan,Volpert} do not directly translate onto our system since they only consider scalar equations and/or linear advection, but could be used to inform a proper analytical undertaking into the wave stability. In addition, it is possible that a bifurcation occurs at some critical $\rho$ changing the stability of the monotone wave, as happens in the Boussinesq model discussed above~\cite{ConstantinKiselevRyzhik, TexierPicardVolpert, TexierPicardVolpert_bifurcation, VladimirovaRosner}.  We leave an investigation of the stability of the waves constructed here and the issues mentioned above to a follow-up work.

\subsection*{Numerical findings}

We now discuss our numerical findings that complement the analysis of system~\eqref{FKPP_Burgers}. To better understand the behaviour of $\rho_\nu$ as $\nu \geq 0$ is varied, we have simulated the full PDE~\eqref{FKPP_Burgers} and determined the asymptotic spreading speed of traveling wave solutions. To do this we have initialized both $T$ and $u$ as Heaviside functions on a spatial domain of $-20\leq x \leq 100$ to allow the wave to travel sufficiently far to the right to observe the asymptotic speed. We set $T = U = 0$ at the right boundary, while setting $T=1$ at the left and using Neumann boundary conditions in $u$. The reason for the discrepancy of boundary conditions at the left is that Theorem~\ref{thm:Existence1} dictates that the waves asymptotically connect to $u_0$ as $x-ct \to -\infty$, which can only be determined upon knowing the exact spreading speed.

The resulting speeds are plotted in Figure~\ref{fig:Speeds} for $\nu = 0,0.1,1,10$, and Figures~\ref{fig:Waves} and~\ref{fig:Waves2} provide profiles from the temporal evolution with $\nu = \rho = 1$. We comment on the fact that we do not have a result that guarantees that such initial conditions asymptotically spread at the minimum wave speed and so the plots in Figure~\ref{fig:Speeds} are only meant as heuristics.

There are a number of conjectures that we can make using these numerics.  First and most importantly, from \Cref{fig:Waves,fig:Waves2}, it appears that the monotone waves connecting $(1,c + \rho - \sqrt{c^2 + \rho^2})$ to $(0,0)$ are stable, at least in the small $\rho$ parameter regime.  Second, we expect that $\nu_*$ in Theorem~\ref{thm:Viscous} is at least 10 and that $\rho_\nu$ monotonically decreases in $\nu > 0$. What is unknown is whether or not $\rho_\nu$ limits to 0 at either a finite value of $\nu > 0$ or as $\nu \to \infty$.  We also observe that the bounds $\rho_0 \leq \frac{16}{3}$ and $\rho_\nu \leq 8$ appear to be overly conservative since $\rho_\nu \leq \rho_0 \approx 2.5$ in Figure~\ref{fig:Speeds}. Finally, we expect that $c_*(\rho)$, $\underline c_*(\nu,\rho)$, and $\overline c_*(\nu,\rho)$ are increasing in both $\nu$ and $\rho$, although a proof of this is notably absent from the results stated above.

\begin{figure} 
\center
\includegraphics[width = 0.49\textwidth]{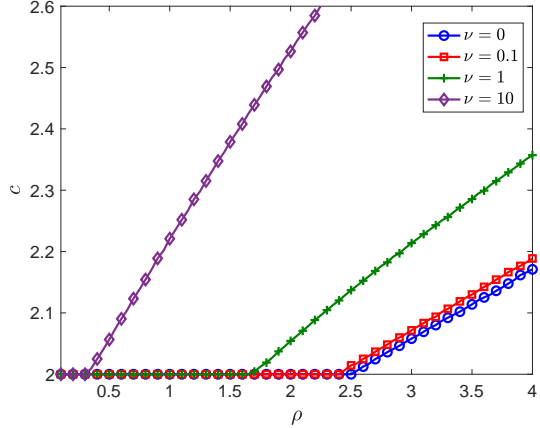} 
\caption{Asymptotic spreading speeds, $c$, observed in numerical integrations of system~\eqref{FKPP_Burgers} for various $\nu \geq 0$.  Notice the clear appearance of threshold $\rho_v$ below which the minimals speed is $2$ and the apparent monotonicity of the spreading speed in both $\nu$ and $\rho$.}
\label{fig:Speeds}
\end{figure} 

\begin{figure} 
\center
\includegraphics[width = 0.49\textwidth]{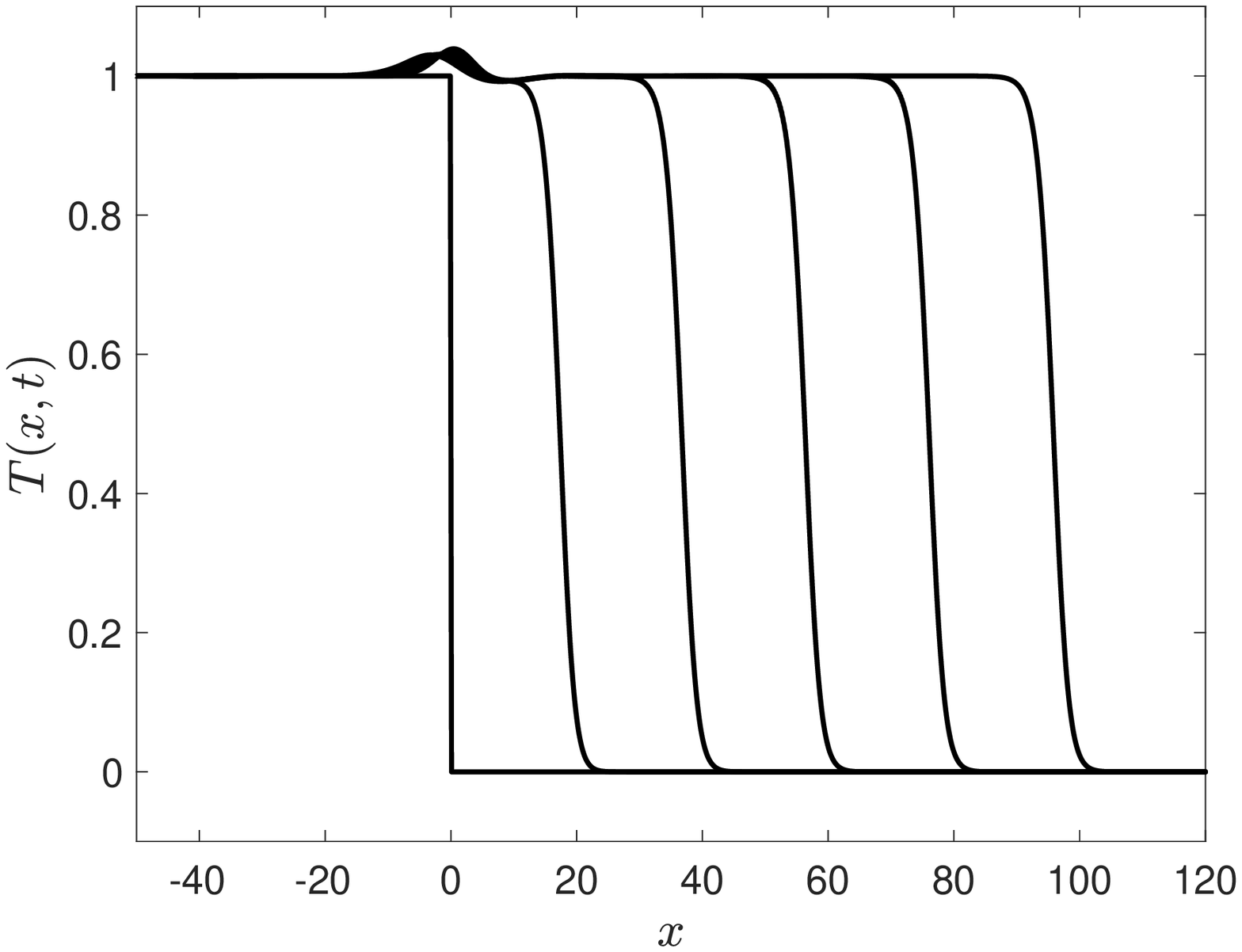}\ \includegraphics[width = 0.49\textwidth]{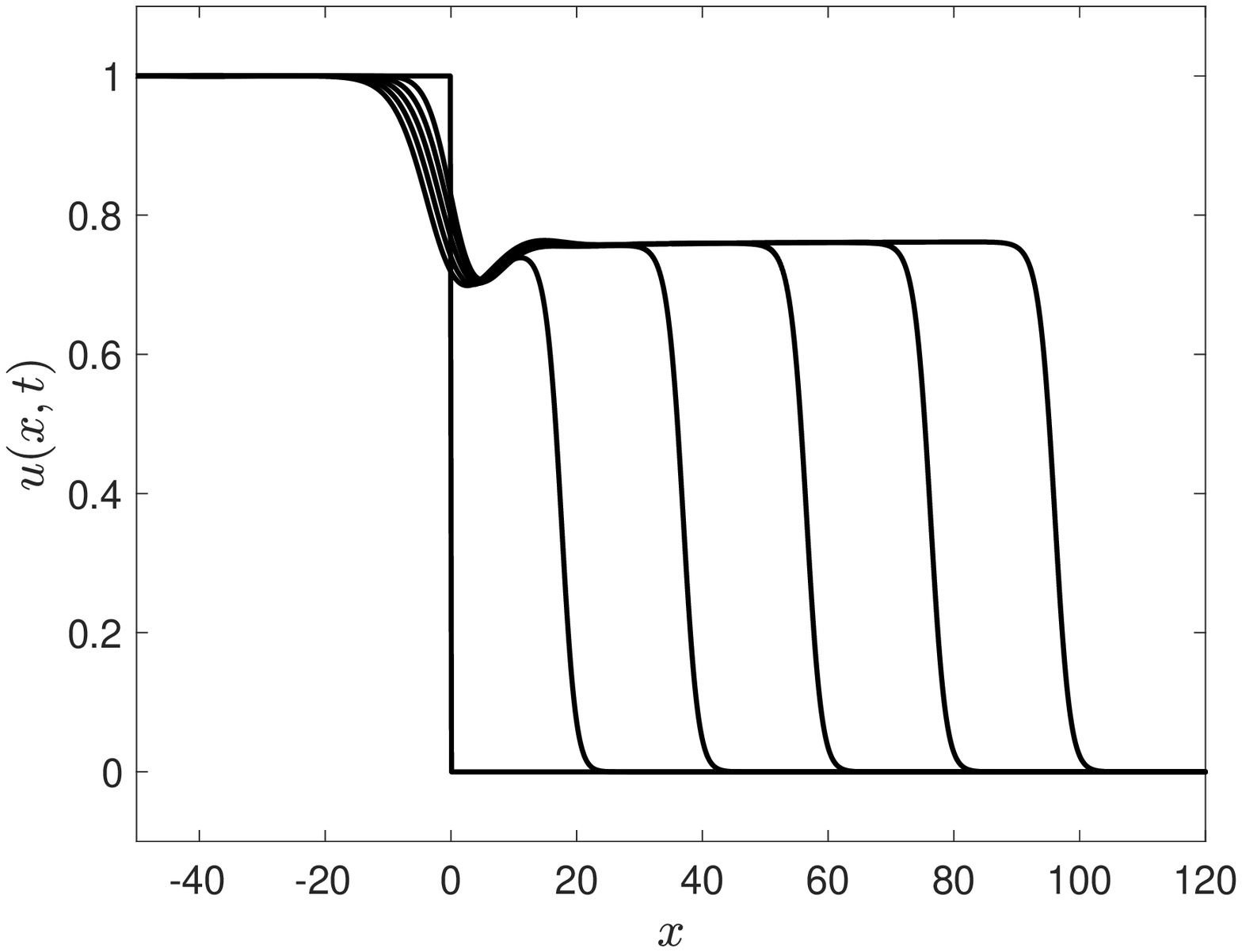}
\caption{Temporal evolution of system~\eqref{FKPP_Burgers} at $\nu = \rho = 1$ with Heaviside initial conditions. Profile snapshots are provided at $t = 0,10,20,30,40,50$ with the waves propagating to the right with asymptotic speed $2$. With these values we have $u_0 \approx 0.76393$, represented at the elongating middle plateau of the $u(x,t)$ component.  It is important to note here that the steady state selected at the back of the emerging wave in $u(x,t)$ corresponds to $u_0 = c + \rho - \sqrt{c^2 + \rho^2}$ instead of $c + \rho + \sqrt{c^2 + \rho^2}$.}
\label{fig:Waves}
\end{figure} 

\begin{figure} 
\center
\includegraphics[width = 0.49\textwidth]{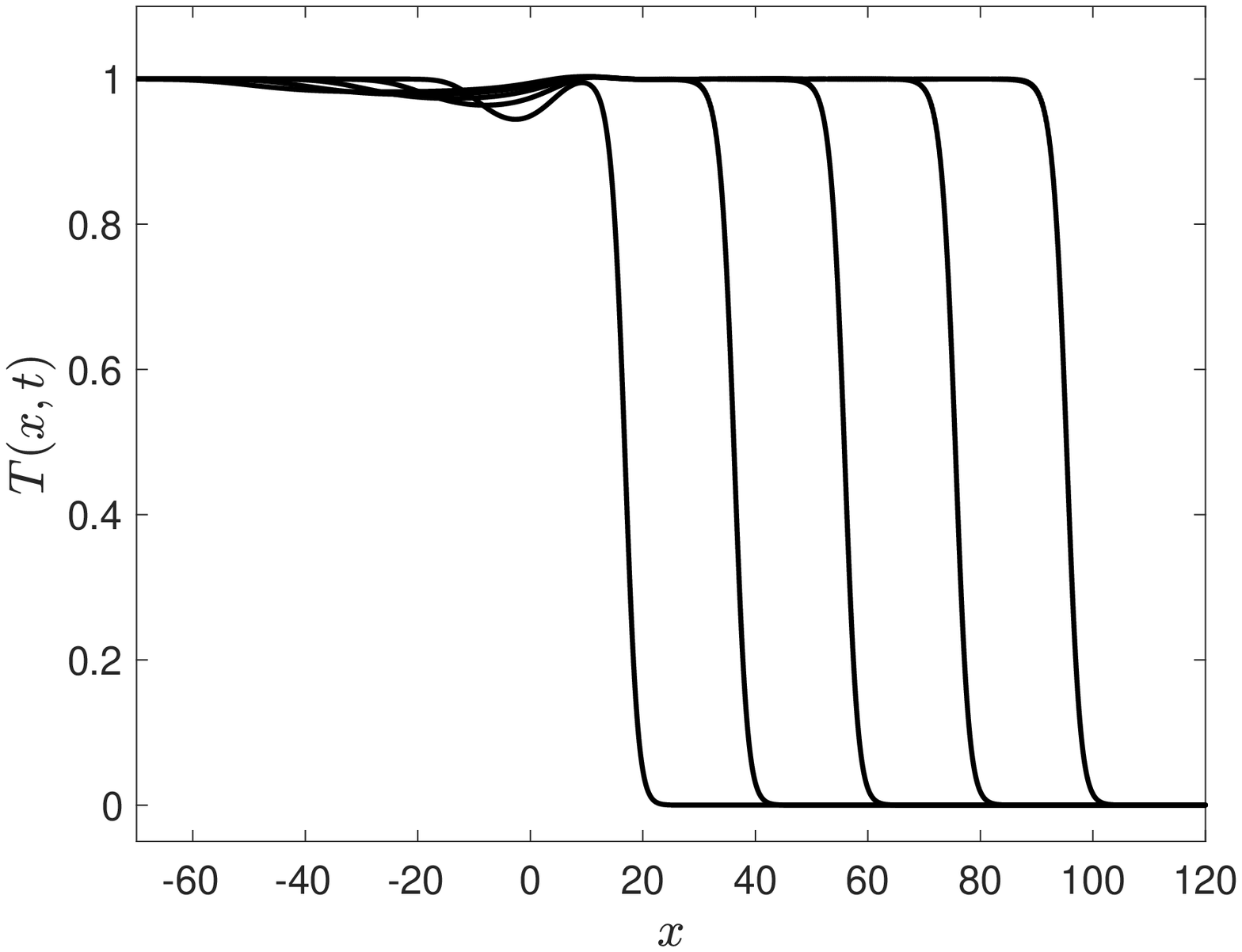}\ \includegraphics[width = 0.49\textwidth]{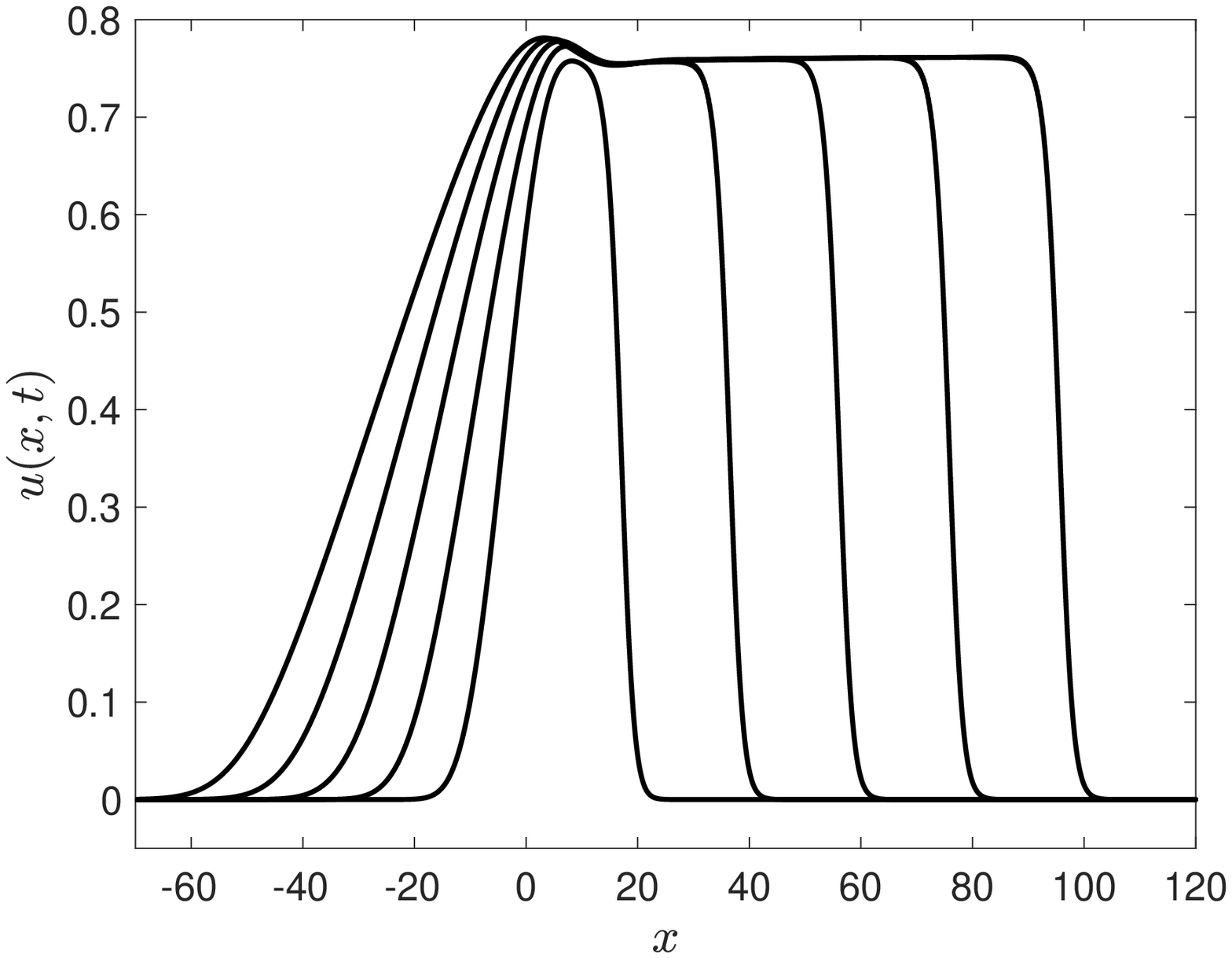}
\caption{Temporal evolution of system~\eqref{FKPP_Burgers} at $\nu = \rho = 1$. $T$ is initialized with a Heaviside initial function while $u$ is initialized to be identically zero. Solutions are again plotted at $t = 10,20,30,40,50$.}
\label{fig:Waves2}
\end{figure} 

In the case when $\rho > 0$ is large, our analytical results dictate that the traveling waves propagate at speeds that are at least proportional to $\rho^{1/3}$. Therefore wave speeds become large when $\rho$ is large and simulating the spreading phenomenon numerically would require significantly larger domains to allow sufficient space for the wave speed to asymptote. Instead of taking such a potentially computationally expensive task, we turn to using the numerical bounding procedures from \cite{BramburgerGoluskin}. In the appendix, we provide detailed information about how the wave speed bounding procedure is applied to~\eqref{FKPP_Burgers}, while here we only discuss the results. In Table~\ref{table1}, we present these upper and lower bounds on the wave speed divided by $\rho^{1/3}$ for various $\nu$ in an effort to determine the asymptotic pre-factor and to provide justification that $\underline c_*(\nu,\rho) = \overline c_*(\nu,\rho)$, at least for large $\rho$. In all cases the bounds are within $1$-$2\%$ of being sharp and the pre-factors have converged to at least two significant digits. From these values we again see significant evidence for the conjecture that the minimum wave speed is monotonically increasing in $\nu \geq 0$. As pointed out above, the lower bound provided by Lemma~\ref{lem:LargeRho2} appears to be sharp since it gives an asymptotic pre-factor for the inviscid equation of at least $(3/2)^{1/3} \approx 1.1447$, which matches our upper and lower bounds to three significant digits. This provides evidence for the conjecture that our lower bound in Lemma~\ref{lem:LargeRho2} is sharp, while the upper bound of $\sqrt{3} \approx 1.7321$ from Lemma~\ref{lem:LargeRho} could be improved. Unfortunately, our upper and lower bounds for the viscous equation, presented in Lemmas~\ref{lem:LargeRho3} and \ref{lem:LargeRho4}, do not appear to be similarly sharp. 

\begin{table} 
\begin{center}
\begin{tabular}{ c | c c c c c c c c c c }
 $\rho$ & $10^{3/2}$ & $10^2$ & $10^{5/2}$ & $10^3$ & $10^{7/2}$ & $10^4$ & $10^{9/2}$ & $10^5$ & $10^{11/2}$ & $10^6$ \\ 
 \hline
 Upper Bound/$\rho^{1/3}$ & 1.194 & 1.169 &1.158 & 1.154 & 1.151 & 1.150 & 1.149 & 1.148 & 1.148 & 1.148 \\  
 Lower Bound/$\rho^{1/3}$ & 1.168 & 1.152 & 1.145 & 1.142 & 1.142 & 1.141 & 1.142 & 1.142 & 1.142 & 1.142 
\end{tabular} \vskip 0.5cm
\begin{tabular}{ c | c c c c c c c c c c }
 $\rho$ & $10^{3/2}$ & $10^2$ & $10^{5/2}$ & $10^3$ & $10^{7/2}$ & $10^4$ & $10^{9/2}$ & $10^5$ & $10^{11/2}$ & $10^6$ \\ 
 \hline
 Upper Bound/$\rho^{1/3}$ & 1.212 &1.190 & 1.180 & 1.176 & 1.173 & 1.172 & 1.171 & 1.171 & 1.171 & 1.170 \\  
 Lower Bound/$\rho^{1/3}$ & 1.159 & 1.140 &1.133 & 1.129 & 1.126 & 1.124 & 1.124 & 1.123 & 1.123 & 1.123 
\end{tabular} \vskip 0.5cm
\begin{tabular}{ c | c c c c c c c c c c }
 $\rho$ & $10^{3/2}$ & $10^2$ & $10^{5/2}$ & $10^3$ & $10^{7/2}$ & $10^4$ & $10^{9/2}$ & $10^5$ & $10^{11/2}$ & $10^6$ \\ 
 \hline
 Upper Bound/$\rho^{1/3}$ & 1.371 &1.357 & 1.351 & 1.349 & 1.347 & 1.346 & 1.346 & 1.346 & 1.346 & 1.346 \\  
 Lower Bound/$\rho^{1/3}$ & 1.343 & 1.331 &1.327 & 1.324 & 1.322 & 1.319 & 1.317 & 1.318 & 1.318 & 1.319 
\end{tabular} \vskip 0.5cm
\begin{tabular}{ c | c c c c c c c c c c }
 $\rho$ & $10^{3/2}$ & $10^2$ & $10^{5/2}$ & $10^3$ & $10^{7/2}$ & $10^4$ & $10^{9/2}$ & $10^5$ & $10^{11/2}$ & $10^6$ \\ 
 \hline
 Upper Bound/$\rho^{1/3}$ & 1.861 & 1.862 & 1.864 & 1.865 & 1.866 & 1.866 & 1.866 & 1.866 & 1.867 & 1.867 \\  
 Lower Bound/$\rho^{1/3}$ & 1.861 & 1.862 & 1.864 & 1.865 & 1.865 & 1.865 & 1.864 & 1.864 & 1.862 & 1.862 
\end{tabular}
\end{center}
\caption{Bounds on the pre-factor of the wave speeds using the bounding methods of \cite{BramburgerGoluskin}. We present upper and lower bounds on the minimum wave speed divided by $\rho^{1/3}$ for system~\eqref{FKPP_Burgers}, with $\nu = 0,0.1,1,10$, from top to bottom.}
\label{table1}
\end{table}

\subsection*{Related work}

Our results are placed into the large body of applied and mathematical studies into the connection between front propagation and advection.  Due to the size of this literature, we are only able to give an incomplete sampling of what has been done and encourage the interested reader to look at references within those cited works below.  We mention the following threads of research: the effect of a large shear flow~\cite{HamelZlatos}, a large periodic flow~\cite{Zlatos, Zlatos2, Zlatos3}, a fractal flow~\cite{MajdaSouganidis_Fractal}, a random spatially invariant flow~\cite{HendersonSouganidis, NolenXin2} and large random shear flow~\cite{NolenXin1}.  We also point out related work in which the reaction-diffusion equation is replaced instead by a Hamilton-Jacobi equation~\cite{XinYuG, XinYuHJ, XinYuZlatos} (see, e.g.,~\cite{MajdaSouganidis} for a discussion of the connection between this and reaction-diffusion equations).  It is also important to mention the large amount of formal work done on the applied end into turbulent combustion~\cite{ASY, ChertkovYakhot, KAW, MayoKerstein1, MayoKerstein2, Yakhot}.

In the above works, the advection is prescribed and is independent of the reacting quantity.  The closest works to the current one are those in which the advection is coupled to the reacting quantity.  The reactive Boussinesq equation and the FKPP-Burgers of Constantin et al. are the closest to the setting investigated here and were discussed in the introduction. Due to Theorems \ref{thm:MinSpeed1} and \ref{thm:Viscous} above, it is natural to conjecture that traveling waves in the reactive Boussinesq model have minimal speed that is $O(\rho^{1/3})$ when $\rho$ is large. Aside from those, the FKPP-Keller-Segel equation~\cite{NadinPerthameRyzhik} is the only other setting we are aware of that has seen significant interest in recent time.  In this equation, the reacting quantity is a population of bacteria that tend to aggregate.  This aggregate is modeled via a coupled advection term (chemotaxis).  However, in contrast to our work, it is known that the chemotactic term (the advection) neither slows down nor speeds up front speeds~\cite{Bramburger, NadinPerthameRyzhik, SalakoShenXue}.  Hence, that phenomena and those studied here are quite different.  We mention, however, various settings in which dispersive chemotaxis can cause faster propagation~\cite{HamelHenderson, Henderson_Chemotaxis}.


\section{Proof of Theorems~\ref{thm:Existence1} and \ref{thm:MinSpeed1}}\label{sec:Inviscid} 

Throughout this section we consider \eqref{FKPP_Burgers} with $\nu = 0$, referred to as the inviscid equation. Traveling wave solutions of \eqref{FKPP_Burgers} take the form $T(x,t) = T(\xi)$ and $U(x,t) = U(\xi)$, where $\xi := x - ct$ is the traveling wave variable and the parameter $c \geq 0$ is the wave speed. With this ansatz $T(\xi)$ and $U(\xi)$ solve
\begin{equation}\label{e.tw}
	\begin{split}
		-c\frac{\drm T}{\drm \xi} - \frac{\drm^2 T}{\drm \xi^2} + \frac{\drm}{\drm \xi}(UT) &= T(1-T), \\
		-c\frac{\drm U}{\drm \xi} + U\frac{\drm U}{\drm \xi} &= \rho T(1 - T),
	\end{split}
\end{equation}
where we recall that the second equation has no second derivative in $\xi$ since $\nu = 0$. Introducing the $\xi$-dependent quantity
\begin{equation}\label{Vdef}
	V := cT + \frac{\drm T}{\drm \xi} - UT 
\end{equation}
results in the first order system of ODE 
\begin{equation}\label{ODE}
	\begin{split}
		\dot{T} &= -cT + UT + V, \\
		\dot{U} &= \bigg(\frac{\rho}{U - c}\bigg)T(1-T), \\
		\dot{V} &= T(T-1).
	\end{split}
\end{equation}
Here and in what follows we use a dot to denote differentiation with respect to the independent variable $\xi$. We note that the existence of traveling wave solutions of \eqref{FKPP_Burgers} with $\nu = 0$ is now equivalent to finding heteroclinic trajectories between two equilibria of \eqref{ODE} such that $T \geq 0$ using the wave speed $c > 0$ as a parameter.

We can simplify our analysis of \eqref{ODE} by observing that it has a conserved quantity; that is, the equality 
\begin{equation}\label{Conserved}
	\frac{1}{2}U^2 - cU + \rho V = C
\end{equation}
holds along trajectories of \eqref{ODE} for some fixed $C \in \R$ and all $\xi \in \R$. Since we are interested in heteroclinic solutions of \eqref{ODE} that satisfy $(T(\infty),U(\infty)) = (0,0)$, it follows from the above conserved quantity that we may restrict ourselves to the curve
\be\label{Conserved_V}
	\frac{1}{2}U^2 - cU + \rho V = 0 \iff V = \frac{U}{2\rho}(2c - U) 	
\ee
representing $C = 0$ in \eqref{Conserved}. This results in the planar dynamical system 
\begin{equation}\label{Planar}
	\begin{split}
		\dot{T} &= -cT + UT + \frac{U}{2\rho}(2c - U) , \\
		\dot{U} &= \bigg(\frac{\rho}{U - c}\bigg)T(1-T), 
	\end{split}
\end{equation}
which forms the basis for investigation in what follows. 

This section is broken down into two subsections. The first, \S~\ref{subsec:InviscidExistence}, provides a quantitative analysis of \eqref{Planar} and details the existence of the desired heteroclinic solutions all $\rho > 0$ and $c > 0$ taken sufficiently large. We also prove the ordering \eqref{Ordering} that must hold for traveling wave of the inviscid equation. We follow this general analysis with \S~\ref{subsec:InviscidSpeed} where we begin by proving the existence of a minimal parameter value, $c_0(\rho)$, for which traveling waves of \eqref{FKPP_Burgers} with $\nu = 0$ exist if and only if $c \geq c_0(\rho)$, for all $\rho > 0$. Upon proving the existence of this minimal wave speed we turn to proving the properties of $c_0(\rho)$ detailed in Theorem~\ref{thm:MinSpeed1} in the small and large $\rho > 0$ parameter regimes.


\subsection{Existence of traveling waves}\label{subsec:InviscidExistence}

In this subsection, we show that, when $c$ is sufficiently large, the planar ODE~\eqref{Planar} exhibits the desired heteroclinic trajectories, thereby showing the existence of traveling wave solutions of \eqref{FKPP_Burgers} with $\nu = 0$ for large $c$. First, we notice that the only equilibrium solutions of~\eqref{Planar} of the form $(1,u_0)$ are those corresponding to
\be
	u_0 = c + \rho \pm \sqrt{c^2 + \rho^2}.
\ee
Our first result shows that there is no heteroclinic connection between $(1, c+ \rho + \sqrt{c^2 + \rho^2})$ and $(0,0)$, the proof of which is left to \Cref{sec:big_eq}.

\begin{lem}\label{lem:big_equilibrium}
	Suppose that $c\geq 1$ and $(T,U)$ is a solution of~\eqref{Planar} such that $(T(-\infty), U(-\infty)) = (1, c + \rho + \sqrt{c^2 + \rho^2})$.  Then, either $T$ becomes negative or $\displaystyle \liminf_{\xi\to\infty} U(\xi) = c$.
\end{lem}

As a result of \Cref{lem:big_equilibrium}, we can restrict our attention to finding heteroclinic orbits connecting $(1,u_0)$ to $(0,0)$, where $u_0 := c + \rho - \sqrt{c^2 + \rho^2}$.  Any such heteroclinic orbit must be on the unstable manifold of $(1,u_0)$ and, as seen in the following proposition, \Cref{prop:Equilibria}, either enters
\be\label{BoxDef}
	\mathcal{B} = \{(T,U):\ 0 \leq T \leq 1,\ 0 \leq U \leq u_0\},
\ee
or the quadrant $\mathcal{Q} = \{(T,U) : T >1, U > u_0\}$.  A simple argument that follows as in the first case of the proof of \Cref{lem:big_equilibrium} shows that the component entering the quadrant $\mathcal{Q}$ cannot connect back to the origin as $\xi \to \infty$.  Therefore, this provides an {\em a priori} justification that we may restrict our attention to heteroclinic orbits that {\em enter} into the box $\mathcal{B}$.

\begin{prop}\label{prop:Equilibria} 
	For each $c,\rho > 0$ we have the following:
	\begin{enumerate}[(1)]
		\item The origin of \eqref{Planar} is locally asymptotically stable for all $c > 0$. In particular, when $c \in (0,2)$ the linearization of \eqref{Planar} about the origin has a pair of complex-conjugate eigenvalues with negative real part, and when $c \geq 2$ all eigenvalues are real and negative.
		\item The unstable manifold of the equilibrium $(1,u_0)$ of \eqref{Planar} is one-dimensional and its tangent vector at this equilibrium points to the interior of $\mathcal{B}$. 
		\item The unstable manifold of $(1,u_0)$ is necessarily monotone while it remains in $\mathcal{B}$. That is, $\dot{T} < 0$ and $\dot{U} < 0$ everywhere along the trajectory in $\mathcal{B}$.
		\item The unstable manifold of $(1,u_0)$ can only leave $\mathcal{B}$ by crossing $U = 0$.
	\end{enumerate}
\end{prop}

\begin{proof}
	The stability of the equilibria is a straightforward checking of the eigenvalues and eigenvectors of the Jacobian matrix evaluated at each of the equilibria, and is therefore omitted. We now prove the monotonicity of the unstable manifold of $(1,u_0)$. We begin by noting that the linearization about the equilibrium $(1,u_0)$ implies that the segment of the unstable manifold that enters in the interior of $\mathcal{B}$ necessarily satisfies $\dot{T} < 0$ initially and we see that any solution which belongs the interior of $\mathcal{B}$ has $\dot{U} < 0$ everywhere since $U < c$ and $0 < T < 1$. Let us then assume for contradiction that there exists some $\xi_0 \in \R$ such that $\dot{T}(\xi_0) = 0$ and $\dot{T}(\xi) < 0$ for all $\xi < \xi_0$. Differentiating the first equation in \eqref{Planar} with respect to $\xi$ and evaluating at $\xi = \xi_0$ gives   
	\[
		\ddot{T}(\xi_0) = -c\dot{T}(\xi_0) + \dot{T}(\xi_0)U(\xi_0) + T(\xi_0)\dot{U}(\xi_0) - \frac{\dot{U}(\xi_0)}{\rho}(U(\xi_0) - c).
	\]
	By assumption we have $\dot{T}(\xi_0) = 0$ and $0 < T(\xi_0) < 1$, thus giving that $\dot{U}(\xi_0) < 0$. Putting this all together gives that 
	\[
		\ddot{T}(\xi_0) = T(\xi_0)\dot{U}(\xi_0) - \frac{\dot{U}(\xi_0)}{\rho}(U(\xi_0) - c)  < 0.
	\] 
	Hence, $T$ achieves a local maximum at $\xi = \xi_0$, which is impossible since $T$ is decreasing for all $\xi < \xi_0$. This proves the point (3) in the lemma.
	
	 Finally, is a straightforward checking that the unstable manifold of $(1,u_0)$ cannot leave $\mathcal{B}$ through the face $T = 0$ since we have $\dot{T} \geq 0$ here. Since the unstable manifold of $(1,u_0)$ is decreasing in $\mathcal{B}$, it follows that it cannot leave through the face $T = 1$ or $U = 1$. Therefore, it can only leave through the face $U = 0$, as claimed.
\end{proof} 

From Proposition~\ref{prop:Equilibria}(1) we can see that for $0 < c < 2$ a heteroclinic orbit (if it exists) would necessarily have negative $T$ and $U$ components. This of course violates the condition that we are seeking nonnegative traveling waves and therefore we have the explicit lower bound $c_* \geq 2$ on the minimum wave speed, as given in Theorem~\ref{thm:MinSpeed1}(1). Notice that Proposition~\ref{prop:Equilibria}(3) also gives the desired monotonicity results from Theorem~\ref{thm:Existence1} whenever a heteroclinic solution exists.

Most important to what follows is that Proposition~\ref{prop:Equilibria}(4) details that the unstable manifold of $(1,u_0)$ can only leave $\mathcal{B}$ by crossing $U = 0$. The following lemma shows that an appropriate trapping boundary can be constructed to prevent $U$ from crossing $0$ for any fixed $\rho > 0$ and all sufficiently large $c$. This result therefore shows that fast-moving ($c \gg 1$) traveling waves of \eqref{FKPP_Burgers} exist for each fixed $\rho > 0$ since any trajectory of \eqref{Planar} that remains in $\mathcal{B}$ for all $\xi\in\R$ must converge to the origin.

\begin{lem}\label{lem:Het} 
	For each $\rho > 0$, there exists a $\overline c \geq 2$, which depends on $\rho$, such that for all $c \geq \overline c$ system \eqref{Planar} has a heteroclinic orbit from $(1,u_0)$ to $(0,0)$ which remains in $\mathcal{B}$ for all $\xi \in \R$.
\end{lem}

\begin{proof}
	To prove this lemma we show that for each fixed $\rho > 0$  the curve
	\begin{equation}\label{Boundary}
		U = c - \sqrt{c^2 - c\rho T}, \quad T \in [0,1] 	
	\end{equation}
	lies in $\mathcal{B}$ and bounds trajectories above it for all $c \gg 1$. We comment on the fact that if \eqref{Boundary} does indeed define a trapping boundary (for the unstable manifold of $(1,u_0)$) that remains in $\mathcal{B}$ for all $\xi \in\R$, it follows from the monotonicity results in Proposition~\ref{prop:Equilibria}(3) that this unstable manifold must converge to the origin as $\xi \to \infty$.
	
	Hence, our first goal is to show that if $c$ is taken large enough, the curve \eqref{Boundary} is contained in $\mathcal{B}$ for all $T \in [0,1]$ and that the unstable manifold of $(1,u_0)$ is initially above it. Using straightforward calculations, we see that
\[\begin{split}
	c - \sqrt{c^2 - c \rho} < u_0
		&\Longleftrightarrow \sqrt{ c^2 + \rho^2} < \rho + \sqrt{c^2 - c\rho}
		\Longleftrightarrow c^2 + \rho^2 < \rho^2 + (c^2 - c\rho) + 2 \rho \sqrt{c^2 - c\rho}\\
		&\Longleftrightarrow \frac{c}{2} < \sqrt{c^2 - c\rho}
		\Longleftrightarrow \frac{c^2}{4} < c^2 - c\rho
		\Longleftrightarrow \rho < \frac{3c}{4}.
\end{split}\]
Hence, as long as $c > 4\rho/3$, we have $c - \sqrt{c^2 + \rho^2} < u_0$ as desired. This implies that we find that the curve \eqref{Boundary} does indeed lie entirely in $\mathcal{B}$ for all $T\in[0,1]$ for all $c$ sufficiently large, relative to $\rho$. We now show that this curve is a trapping boundary.
	
	For this curve to be a trapping boundary we require that on the curve \eqref{Boundary} we have 
	\begin{equation}\label{Trap}
		\frac{\drm}{\drm \xi}\bigg(\frac{U}{2\rho}(2c - U) - \frac{c}{2} T\bigg) \geq 0.
	\end{equation} 
	Indeed, rearranging \eqref{Boundary} we have 
	\begin{equation}
		\frac{U}{2\rho}(2c - U) - \frac{c}{2} T = 0,
	\end{equation} 
	and therefore \eqref{Trap} defines a condition for which trajectories of \eqref{Planar} that initially start above \eqref{Boundary} in $\mathcal{B}$ must remain above it for all $\xi$. Restricting $T$ and $U$ to satisfy \eqref{Boundary} and evaluating gives
	\begin{equation}\label{e.c9101}
		\begin{split}
		\frac{\drm}{\drm \xi}\bigg(\frac{U}{2\rho}(2c - U) - \frac{c}{2} T\bigg) &= \frac{1}{\rho}(c - U)\dot{U} - \frac{c}{2}\dot{T} \\
		&= T(T-1) -\frac{c}{2}(-cT + UT + \frac{c}{2} T) \\
		&= T\bigg[-\frac{c^2}{4} - \frac{c}{2} (U - c) + (T - 1)\bigg] \\
		&= T\bigg[-\frac{c^2}{4} + \frac{c}{2}\sqrt{c^2 - c\rho T} + (T - 1)\bigg]. 
		\end{split}
	\end{equation}
	For the above expression to be nonnegative we require that
	\begin{equation}\label{TrapCondition}
		\frac{c^2}{4} - \frac{c}{2}\sqrt{c^2 - c\rho T} + (1 - T) \leq 0
	\end{equation}
	for all $T \in [0,1]$. Again taking $c \gg 1$, we have that \eqref{TrapCondition} is bounded as follows for all $T \in [0,1]$:
	\begin{equation}
		\begin{split}
		 	\frac{c^2}{4} - \frac{c}{2}\sqrt{c^2 - c\rho T} + (1 - T) &\leq \frac{c^2}{4} - \frac{c}{2}\sqrt{c^2 - c\rho} + 1 \\
			&= -\frac{c^2}{4} + \frac{c\rho}{4} + 1 + \mathcal{O}(\rho^2),	
		\end{split}
	\end{equation} 
	which for each fixed $\rho > 0$ is strictly negative for all sufficiently large $c > 0$. This therefore shows that the curve \eqref{Boundary} forms a trapping boundary when $c$ is taken sufficiently large, and therefore trajectories in $\mathcal{B}$ cannot cross this curve. From the previous discussion, this implies that for fixed $\rho > 0$ and sufficiently large $c$ we have that the unstable manifold of $(1,u_0)$ converges to the origin as $\xi \to \infty$, thus giving a heteroclinic orbit. This completes the proof.
\end{proof} 

Having now demonstrated the existence of traveling waves, we provide the following auxiliary result which establishes the ordering \eqref{Ordering} and is useful for showing that the set of admissible traveling wave speeds makes up a closed half-line in the following subsection. 

\begin{lem}\label{lem:Ordering} 
	Fix any $c,\rho>0$ and suppose that $(T, U)$ solves~\eqref{Planar} with $(T(-\infty),U(-\infty)) = (1,u_0)$.  If $(T(\xi),U(\xi)) \in \mathcal{B}$ for all $\xi < \xi_0$ then $\rho T(\xi_0) < U(\xi_0)$.
\end{lem}
\begin{proof}
	By the arguments in \Cref{prop:Equilibria},
	we have that $\dot{T}(\xi_0) < 0$. 
	Hence, from~\eqref{Planar}, we have
	\be
		0
			< - \dot{T}
			\leq (c- U)T - \frac{U}{\rho}\left( c - \frac{U}{2}\right)
			\leq (c-U)T - \frac{U}{\rho}(c - U).
	\ee
	Using the fact that $U \leq u_0 < c$, we arrive at the ordering $\rho T(\xi_0) < U(\xi_0)$ by simply rearranging the above and dividing by $c-U$.
\end{proof} 

\subsubsection{Proof of \Cref{lem:big_equilibrium}: ruling out the other unstable equilibrium}\label{sec:big_eq}

\begin{proof}[Proof of \Cref{lem:big_equilibrium}]
The proof when $(T,U)$ is an equilibrium is trivial, and, hence, omitted.  Using the linearization principle, either $\dot T(\xi)$ is positive or negative for all sufficiently negative $\xi$.

{\bf Case one: $\dot T(\xi)$ is negative for sufficiently negative $\xi$.} 
One can readily check that then $\dot U(\xi)$ is positive whenever $T \in (0,1)$.  If $T$ becomes negative, the proof is finished.  Hence, we assume that $T\geq 0$ everywhere.

Next, we rule out that $T\geq 1$ ever via an argument by contradiction.  If this were to happen, it must be that $T$ has a local minimum at $\xi_0$ such that $T(\xi_0) \in [0,1)$.  We consider first the case where $T(\xi_0) > 0$.  Then, differentiating~\eqref{Planar}, we find
\be\label{e.second_derivative}
	\ddot{T}(\xi_0) = -c\dot{T}(\xi_0) + \dot{T}(\xi_0)U(\xi_0) + T(\xi_0)\dot{U}(\xi_0) - \frac{\dot{U}(\xi_0)}{\rho}(U(\xi_0) - c).
\ee
Since $\xi_0$ is the location of a minimum, we have that the left hand side is nonnegative and $\dot T(\xi_0) = 0$.  It follows that
\be\label{e.second_derivative2}
	0 \leq T(\xi_0)\dot{U}(\xi_0) - \frac{\dot{U}(\xi_0)}{\rho}(U(\xi_0) - c)
		= -\dot U(\xi_0) \left(U(\xi_0) - c - T(\xi_0)\right).
\ee
As $T \in (0,1)$ for all $\xi < \xi_0$ and, thus, $\dot U(\xi) > 0$ for all $\xi \in (-\infty,\xi_0]$, it follows that $U(\xi_0) > U(-\infty) > 2c\geq c + T(\xi_0)$, since we have assumed that $c \geq 1$.  The combination of this and the positivity of $\dot U$, implies that the right hand side in~\eqref{e.second_derivative2} is negative, which is a contradiction.

Next, we consider the case where $T(\xi_0) = 0$.  In this case, we have
\be
	0 = \dot T(\xi_0)
		= \frac{U(\xi_0)}{2\rho} \left(2c - U(\xi_0)\right).
\ee
Arguing as in the previous paragraph, we have $U(\xi_0) > 2c$, which implies that the right hand side above is negative.  Hence, at $\xi = \xi_0$, $T$ transversely crosses from positive to negative, and therefore is negative for $\xi$ in an open neighbourhood to the right of $\xi_0$.  This is a contradiction since we assumed $T \geq 0$ for all $\xi \in \R$. This concludes the proof for the first case.

{\bf Case two: $\dot T(\xi)$ is positive for sufficiently negative $\xi$.}

We consider first the case that it is positive.  Notice that, as a result, $T>1$ on a set $(-\infty, \xi_0)$, where we allow for the possibility $\xi_0 = \infty$.

We first suppose that $\xi_0 = \infty$.  In this case, define the auxiliary function $Y = (U-c)^2$.  We see that
 \[
 	\dot Y = 2(U-c) \dot U = 2 \rho T(1-T),
\]
and, hence, $Y$ is decreasing everywhere.  As $Y\geq 0$, by definition, we find that $U$ cannot have crossed $c$, as this would have caused $Y$ to increase.  It follows that $\displaystyle \liminf_{\xi\to\infty}U(\xi) \geq c$.

Next, consider the case where $\xi_0 < \infty$.  There are two subcases: either $U$ crosses $c$ or not on $(-\infty, \xi_0)$.

Let us consider the first subcase; that is, suppose that there exists $\xi_1 < \xi_0$ such that $U(\xi_1) = c$.  Let $Y$ be as above and notice that, $\dot Y(\xi_1) < 0$ and $Y(\xi_1) = 0$, implying that $Y$ becomes negative.  This is clearly a contradiction, which implies that this subcase may not occur.

We now consider the second subcase; that is, suppose that $U(\xi_0) \geq c$.  Since $T>1$ on $(-\infty, \xi_0)$ and $T(\xi_0) = 1$, it follows that $\dot T(\xi_0) \leq 0$.  Using~\eqref{Planar}, we have
\[
	0 = \dot T(\xi_0)
		\geq - c T + UT + \frac{U}{2\rho}\left(2c - U\right)
		= -c + U + \frac{U}{2\rho}\left(2c - U\right).
\]
Recalling that $T>1$ on $(-\infty,\xi_1]$, we see that $U$ is decreasing on $(-\infty, \xi_0)$, which implies that $U < c+\rho + \sqrt{c^2+\rho^2}$.  Since $U(\xi_0)$ is between the roots of the polynomial (in $U$) in the right hand side above, it follows that the right hand side above is positive.  This is a contradiction, which finishes the proof. 

\end{proof}


\subsection{The minimum wave speed}\label{subsec:InviscidSpeed}

In this section we prove that for all $\rho > 0$ there exists a closed half-line of speeds with minimum wave speed $c_*(\rho) \geq 2$ of nonnegative traveling waves to \eqref{FKPP_Burgers} when $\nu = 0$. In addition, we obtain estimates on the behavior of $c_*(\rho)$ near $\rho =0$ and when $\rho \gg 1$.

\subsubsection{A half-line of admissible speeds}

We begin by establishing that the admissible speeds of traveling waves make up a half-line $[c_*(\rho),\infty)$.  The main step in this direction is to establish the following lemma, which states that if a wave exists with a particular speed $c_1$, then a wave exists with any speed $c>c_1$.   In the previous section, we showed that traveling waves exist for $c$ sufficiently large and do not exist for $c <2$.  Hence, the lemma below shows that the set of admissible speeds is has the form $(c_*(\rho),\infty)$ or $[c_*(\rho),\infty)$ for some $c_*(\rho) \geq 2$.

\begin{lem} \label{lem:MinSpeed} 
Fix any $\rho, \underline c>0$ and suppose that $(\underline T, \underline U) \in\mathcal{B}$ is a heteroclinic trajectory of \eqref{Planar} from $(1,u_0)$ to $(0,0)$ with $c = \underline c$. Then, for any $\overline c> \underline c$, there exists a heteroclinic trajectory $(\overline T, \overline U) \in \mathcal{B}$ with $c = \overline c$.
\end{lem}

\begin{proof}
	Let $(\overline T, \overline U)$ denote the component of the unstable manifold of $(1,u_0)$ that enters into $\mathcal{B}$. From our work in the previous subsection, as long as $(\overline T, \overline U)$ remains in the box $\mathcal{B}$, $(\overline T, \overline U) \to (0,0)$ as $\xi \to \infty$ and is, thus, a desired heteroclinic trajectory of \eqref{Planar}. In addition, Proposition~\ref{prop:Equilibria}(4) details that $(\overline T, \overline U)$ can only initially exit the box $\mathcal{B}$ through the face $\{(T,0) : 0 \leq T \leq 1\}$.  In summary, it is enough to show that $\overline U > 0$ for all $\xi\in\R$. In order to establish this, we show that $(\overline T, \overline U)$ cannot cross the trajectory $(\underline T, \underline U)$, which establishes the positivity of $\overline U$ via the positivity of $\underline U$.  We argue by contradiction, assuming that $(\overline T, \overline U)$ is not always above $(\underline T, \underline U)$.
	
	We begin by noting that
	\begin{equation}
		\overline U(-\infty) = \overline c + \rho - \sqrt{ \overline c^2 + \rho^2}
			> \underline c + \rho - \sqrt{\underline c^2 + \rho^2}
			= \underline U(-\infty).
	\end{equation}
	The inequality here is due to the fact that $c + \rho - \sqrt{ c^2 + \rho^2}$ is strictly increasing as a function of $c>0$. Hence, $(\overline U, \overline T)$ is initially above $(\underline T, \underline U)$ and so, in what follows we assume that there exists some $\xi_0\in\R$ such that
	\be
		(\overline T(\xi_0), \overline U(\xi_0)) = (\underline T(\xi_0), \underline U(\xi_0)) \in\mathcal{B}
	\ee
	and $(\overline T(\xi), \overline U(\xi))$ always lies above the curve $\{(\underline T(\xi), \underline U(\xi)) : \xi \in (-\infty, \xi_0)\}$ for all $\xi < \xi_0$.  Some simple vector calculus and the non-positivity of $\dot{\overline T}$, $\dot{\overline U}$, $\dot{\underline T}$, and $\dot{\underline U}$ implies that
	\be\label{e.c7162}
		\frac{\dot{\overline U}(\xi_0)}{\dot{\overline T}(\xi_0)}
			\geq \frac{\dot{\underline U}(\xi_0)}{\dot{\underline T}(\xi_0)}.
	\ee
	For the ease of notation, let us write 
	\be
		T = \overline T(\xi_0) = \underline T(\xi_0)
			\quad \text{ and }\quad
		U = \overline U(\xi_0) = \underline U(\xi_0).
	\ee
	Then, using~\eqref{Planar},~\eqref{e.c7162} becomes
	\be
		\frac{\rho T(1-T)}{(\overline c - U)\left[ (\overline c - U) T - \frac{U}{2\rho}\left(2\overline c - U\right)\right]}
			\geq \frac{\rho T(1-T)}{(\underline c - U)\left[ (\underline c - U) T - \frac{U}{2\rho}\left(2\underline c - U\right)\right]}.
	\ee
	which is equivalent to
	\be\label{e.c7163}
		(\overline c - U)\left[ (\overline c - U) T - \frac{U}{2\rho}\left(2\overline c - U\right)\right]
			\leq (\underline c - U)\left[ (\underline c - U) T - \frac{U}{2\rho}\left(2\underline c - U\right)\right].
	\ee
	We now show that this cannot hold and so the point $\xi_0$ cannot exist, completing the proof.
	
	Let us define 
	\begin{equation}
		F(c) = (c - U)\left[ ( c - U) T - \frac{U}{2\rho}\left(2c - U\right)\right].
	\end{equation}  
	Notice that proving $F$ is strictly increasing implies that \eqref{e.c7163} cannot be true, and therefore the proof is finished. To show this we differentiate $F$ with respect to $c$ to obtain
	\be
		\frac{\drm F}{\drm c}
			= \left[ (c - U) T - \frac{U}{2\rho}\left(2 c - U\right)\right]
				+ (c-U) \left[ T - \frac{U}{\rho} \right].
	\ee
	The second term above is positive since $c> U$ for any $c \geq \underline{c}$ and from Lemma~\ref{lem:Ordering} we have that $T > U/\rho$.  On the other hand, we have that the first term is positive when $c = \underline c$ since this is equal to $-\dot{\underline T}(\xi_0)$, which is positive from Proposition~\ref{prop:Equilibria}(3); that is
	\be
		\left[ (\underline c - U) T - \frac{U}{2\rho}\left(2\underline c - U\right)\right]
			= - \dot{\underline T}(\xi_0)
			> 0.
	\ee
	Thus, we conclude that $F' > 0$, which from our arguments above shows that \eqref{e.c7163} cannot be true, and therefore completes the proof.
\end{proof} 

We now show that there is a traveling wave with speed $c_*(\rho)$, which shows that the set of admissible speeds is of the form $[c_*(\rho),\infty)$.

\begin{cor}\label{cor:MinSpeed}
	For each $\rho > 0$ there exists $2 \leq c_*(\rho) < \infty$ such that \eqref{Planar} has a unique heteroclinic trajectory from $(1,u_0)$ to $(0,0)$ if and only if $c \geq c_*(\rho)$.
\end{cor}

\begin{proof}
	As discussed above, the arguments in \Cref{lem:MinSpeed}, \Cref{lem:Het}, and \Cref{prop:Equilibria}(1) combine to give that the set of admissible speeds is either $(c_*(\rho), \infty)$ or $[c_*(\rho),\infty)$ for some $c_*(\rho) \geq 2$.  Hence, it is enough to show that there is a traveling wave with speed $c_*(\rho)$ in order to conclude.
But, if we assume that at $c = c_*(\rho)$ there is not a heteroclinic trajectory, Proposition~\ref{prop:Equilibria} implies that the $U$-component unstable manifold of $(1,u_0)$ must become negative at some point. A simple limiting argument using the heteroclinic trajectories for any $c > c_*(\rho)$ (for which the $U$-components are non-negative) gives that this cannot be possible, and hence the result of the corollary follows. 
\end{proof}

\subsubsection{The behavior of the minimal speed for small and large $\rho$}

Having now characterized the set of admissible traveling wave speeds of
\eqref{FKPP_Burgers} as a closed half-line, we now seek to prove the bounds on the minimal speed $c_*(\rho)$ for small and large $\rho>0$. We begin with the following result pertaining to small $\rho > 0$.

\begin{lem}\label{lem:SmallRho} 
	If $0 < \rho \leq 1$ and $c \geq 2$, then system \eqref{Planar} has a heteroclinic orbit from $(1,u_0)$ to $(0,0)$ which remains in $\mathcal{B}$ for all $\xi \in \R$. As a consequence, $c_*(\rho) = 2$ for all $\rho \in (0,1]$.
\end{lem}
\begin{proof}
We argue as in the proof of Lemma~\ref{lem:Het} but now restrict ourselves to the small $\rho > 0$ parameter regime. Again we consider the trapping curve (cf.~\eqref{Boundary})
	\be\label{Boundary1}
		U = 2 - \sqrt{4 - 2\rho T},
			\qquad T \in [0,1]
	\ee
	and show that for $\rho \in(0,1)$ this curve forms a trapping boundary with $c = 2$. In view of \Cref{prop:Equilibria}, which implies that no traveling wave solutions exist for $0 < c < 2$, and \Cref{lem:MinSpeed}, this shows that $c_*(\rho) = 2$ for all $\rho \in (0,1)$.
	
	From the arguments in \Cref{lem:Het}, we see that, as long as $\rho < 3/2$ (recall that $c = 2$, here), the component of the unstable manifold of $(1,u_0)$ that enters into $\mathcal{B}$ initially lies above the curve~\eqref{Boundary1}, that is
	\be
		U(-\infty)
			= u_0
			> 2 - \sqrt{ 4 - 2 \rho}
			= 2 - \sqrt{4 - 2 \rho T(-\infty)}.
	\ee
	It therefore remains to show that the solution $(T,U)$ of~\eqref{Planar} cannot cross the curve~\eqref{Boundary1} in $\mathcal{B}$.  Recall that due to \Cref{prop:Equilibria}(4), this shows that $(T,U)$ cannot exit the box $\mathcal{B}$. Hence, from the monotonicity of $(T,U)$, in $\mathcal{B}$, it follows that $(T,U)$ is a heteroclinic connection between $(1,u_0)$ and $(0,0)$ with $c=2$, which would finish the proof.  We now establish this.

	First note that~\eqref{Boundary1} is the lower solution of
	\be
		\frac{U}{2\rho}\left(4 - U\right) - T = 0.
	\ee
	Hence, if $(T,U)$ were to cross the curve~\eqref{Boundary1} at some point $\xi_0 \in \R$, it must be that $\frac{U(\xi)}{2\rho}\left(4 - U(\xi)\right) - T(\xi)$ goes from positive to negative at $\xi_0$.  This yields that, at $\xi_0$,
	\be\label{e.crossing}
		\frac{\drm}{\drm\xi}\left( \frac{U}{2\rho}\left(4 - U\right) - T\right)
			\leq 0.
	\ee
	We show that this inequality cannot hold.  Indeed, if $\xi_0$ exists then
	\be\label{SmallRhoBndry}
	\begin{split}
		0 \geq \frac{\drm}{\drm\xi}\left( \frac{U}{2\rho}\left(4 - U\right) - T\right)
			&= \frac{1}{\rho}\left(2 - U\right)\dot U - \dot T
			= -T(1-T) + T(2-U) - \frac{U}{2\rho}\left(4-U\right)\\
			&= T \left( T - U\right).
	\end{split}
	\ee
	In the last equality, we used that $(U/2\rho)(4-U) = T$ at $\xi_0$. By \Cref{lem:Ordering}, we have that $U \leq \rho T$. Putting this into \eqref{SmallRhoBndry}, we find
	\[
		\frac{\drm}{\drm\xi}\left( \frac{U}{2\rho}\left(4 - U\right) - T\right) = T \left( T - U\right)
			\geq T(T - \rho T) > 0,
	\]
	since $\rho < 1$, which concludes the proof in this regime. The case $\rho=1$ can be obtained through a simple limiting argument that we omit. This finished the proof.
\end{proof} 

We now move to the large $\rho > 0$ parameter regime to show that $c_0(\rho) \sim \rho^{1/3}$ as $\rho \to \infty$. In what follows, Lemma~\ref{lem:LargeRho} proves that an upper bound on $c_0(\rho)$ in terms of $\rho^{1/3}$ exists and Lemma~\ref{lem:LargeRho2} complements this result with a lower bound that is of the form $\rho^{1/3}$ as well. These results combined prove Theorem~\ref{thm:MinSpeed1}(3). 

\begin{lem}\label{lem:LargeRho} 
	Fix any $\sigma > \sqrt 3$.  There exists $\rho_0 > 0$ such that for all $\rho \geq \rho_0$, system \eqref{Planar} has a heteroclinic orbit from $(1,u_0)$ to $(0,0)$ which remains in $\mathcal{B}$ for all $\xi \in \R$ and $c \geq \sigma_1\rho^{1/3}$. As a consequence,
	\[
		\limsup_{\rho\to\infty} \frac{c_*(\rho)}{\rho^{1/3}} \leq \sqrt 3.
	\]
\end{lem}

\begin{proof}
	We proceed similarly as in the proof of Lemmas~\ref{lem:Het} and~\ref{lem:SmallRho} by constructing an appropriate trapping region.  We now take $c = \sigma\rho^\frac{1}{3}$ with a some $\sigma > \sqrt{3}$ and consider the curve 
	\begin{equation}\label{Region1_inv}
		U = \sigma\rho^\frac{1}{3} - \rho^\frac{1}{3}\sqrt{\sigma^2 - 2T}, \quad T \in [0,1], 	
	\end{equation}
	which can equivalently be written as the lower branch of the curve 
	\be
		\frac{U}{2\rho}(2\sigma\rho^\frac{1}{3} - U) - \rho^{-\frac{1}{3}} T = 0.
	\ee
	Note the subtle difference in the curves~\eqref{Boundary} (see also~\eqref{Boundary1}) and~\eqref{Region1_inv}: crucially, the coefficient in front of the $T$ term changes. We show that when $\rho \gg 1$ the curve \eqref{Region1_inv} forms a trapping boundary inside of $\mathcal{B}$. 
	
	We first confirm that this curve lies in $\mathcal{B}$ for all $T \in [0,1]$ by checking that $\sigma\rho^\frac{1}{3} - \rho^\frac{1}{3}\sqrt{(\sigma^2 - 2)} < u_0$. We have
	\begin{equation}\label{LargeRho1}
		\begin{split}
			\sigma\rho^\frac{1}{3} - \rho^\frac{1}{3}\sqrt{(\sigma^2 - 2)} - u_0  &= -\rho - \rho^\frac{1}{3}\sqrt{(\sigma^2 - 2)} + \sqrt{\sigma^2\rho^\frac{2}{3} + \rho^2} \\
			&= -\rho - \rho^\frac{1}{3}\sqrt{(\sigma^2 - 2)} + \rho\sqrt{1 + \sigma^2\rho^{-\frac{4}{3}}} \\
			&= -\rho - \rho^\frac{1}{3}\sqrt{(\sigma^2 - 2)} + \rho + \frac{\sigma^2}{2\rho^\frac{1}{3}}  + \mathcal{O}\bigg(\frac{1}{\rho^\frac{5}{3}}\bigg) \\
			&= - \rho^\frac{1}{3}\sqrt{(\sigma^2 - 2)} + \frac{\sigma^2}{2\rho^\frac{1}{3}}  + \mathcal{O}\bigg(\frac{1}{\rho^\frac{5}{3}}\bigg)
		\end{split}
	\end{equation}
	Since $\sigma > \sqrt{3}$ is fixed, it follows that for sufficiently large $\rho > 0$ we have that \eqref{LargeRho1} is negative, verifying that the curve \eqref{Region1_inv} lies in $\mathcal{B}$ for all $T \in[0,1]$.
	
	We now verify that \eqref{Region1_inv} is a trapping boundary in $\mathcal{B}$ when $\rho$ is taken large enough. Following the proof of Lemma~\ref{lem:Het} (see~\eqref{e.c9101}), it follows that the trapping boundary condition is equivalent to having 
	\begin{equation}
		\frac{1}{\rho^\frac{2}{3}} - \sqrt{(\sigma^2 - 2T)} + (1 - T) \leq 0
	\end{equation}  
	for all $T \in [0,1]$. Then, using that $0 \leq T \leq 1$, we find
	\begin{equation}
		\frac{1}{\rho^\frac{2}{3}} - \sqrt{(\sigma^2 - 2T)} + (1 - T) \leq \frac{1}{\rho^\frac{2}{3}} - \sqrt{(\sigma^2 - 2)} + 1.
	\end{equation}  
	As $\sigma >\sqrt{3}$, the right hand side above is negative if $\rho$ is sufficiently large. This concludes the proof.
\end{proof} 

\begin{lem}\label{lem:LargeRho2} 
	The system \eqref{Planar} does not have a heteroclinic orbit from $(1,u_0)$ to $(0,0)$ that remains in $\mathcal{B}$ for all $\xi \in \R$ for any $c < (3\rho / 2)^{1/3}$.  As a consequence,
	\[
		c_*(\rho) \geq \min\left\{2, \left(\frac{3}{2}\right)^{1/3} \rho^{1/3}\right\}
			\quad \text{ for all }\rho \geq 0.
	\]
\end{lem}

\begin{proof}
	We prove this by assuming that we have a heteroclinic orbit from $(1,u_0)$ to $(0,0)$ remaining in $\mathcal{B}$ and obtain a lower bound for $c$. The lower bound of $2$ follows from \Cref{prop:Equilibria}, so we need only obtain the $\mathcal{O}(\rho^{1/3})$ lower bound claimed above. We begin with the fact that $V \geq 0$ (see~\eqref{Conserved_V}) and $1 - T \geq 0$ since $(T,U) \in \mathcal{B}$, by assumption.  Hence,
	\[
		0 \leq \int (1-T) V d \xi,
	\]
	where the limits of integration are always $-\infty$ and $\infty$, respectively. On the other hand, due to~\eqref{Vdef}, we can re-write this as
	\[
		0
			\leq \int (1- T) \left[ \dot T + (c-U) T\right] d\xi
			\leq \int \frac{d}{d\xi} \left(T - \frac{T^2}{2}\right) d \xi  + \int (c-U) T(1-T) d\xi.
	\]
	The first term on the right is readily computed to be $-1/2$.  Rearranging the above yields
	\be\label{e.c831}
		\frac{1}{2}
			\leq \int (c-U) T(1-T) d\xi.
	\ee
	
	We now aim to compute the right hand side of~\eqref{e.c831}.  Using~\eqref{Planar}, we find
	\[
		\int (c-U) T(1-T) d\xi
			= -\frac{1}{\rho} \int (c-U)^2 \dot U d \xi
			= \frac{1}{3\rho} \int \frac{d}{d \xi} (c-U)^3 d\xi
			= \frac{1}{3\rho} \left( c^3 - (c-u_0)^3 \right).
	\]
	Hence,~\eqref{e.c831} becomes
	\[
		\frac{3\rho}{2}
			\leq c^3 - (c - u_0)^3
			\leq c^3.
	\]
	Taking the cubed root of both sides yields the claim.
\end{proof} 


\section{Proof of Theorem~\ref{thm:Viscous}}\label{sec:Viscous} 

In this section, we consider the existence and further properties of traveling wave solutions to \eqref{FKPP_Burgers} with $\nu > 0$, the aforementioned viscous equation. As in the previous section, we consider $\rho > 0$ and substitute the ansatz $T(x,t) = T(x-ct)$ and $u(x,t) = U(x-ct)$ into \eqref{FKPP_Burgers} to arrive at 
\begin{equation}\label{TWave2}
	\begin{split}
		-c\frac{\drm T}{\drm \xi} - \frac{\drm^2 T}{\drm \xi^2} + \frac{\drm}{\drm \xi}(UT) &= T(1-T), \\
		-c\frac{\drm U}{\drm \xi} - \nu\frac{\drm^2 U}{\drm \xi^2}  + U\frac{\drm U}{\drm \xi} &= \rho T(1 - T),
	\end{split}
\end{equation}
where as in the previous section $\xi := x-ct$ is the traveling wave independent variable. Let us define the $\xi$-dependent quantities
\be
	\begin{split}
		V &:= cT + \frac{\drm T}{\drm \xi} - UT, \\ 
		W &:= cU + \nu\frac{\drm U}{\drm \xi} - \frac{1}{2}U^2,
	\end{split}
\ee
to transform \eqref{TWave2} to the first order system of equations
\begin{equation}\label{ODE2}
	\begin{split}
		\dot{T} &= -cT + UT + V, \\
		\dot{U} &= \frac{1}{\nu}\bigg[-cU + \frac{1}{2}U^2 + W\bigg], \\
		\dot{V} &= T(T-1), \\
		\dot{W} &= \rho T(T-1).
	\end{split}
\end{equation}
In system \eqref{ODE2} we have that a traveling wave solution of \eqref{FKPP_Burgers} as described in Theorem~\ref{thm:Viscous} corresponds to a heteroclinic trajectory connecting the equilibrium $(1,u_0,c-u_0,\rho(c-u_0))$ to the origin. 

Much like system~\eqref{ODE}, system~\eqref{ODE2} has a conserved quantity. Precisely, along trajectories of \eqref{ODE2} we have
\be
	\rho V - W = C,
\ee
for a fixed constant $C\in\R$ and all $\xi\in\R$. Since we are searching for solutions which asymptotically approach the origin, we may restrict ourselves to $C = 0$ and reduce the dimension of the system \eqref{ODE2} by taking $W = \rho V$. In this case \eqref{ODE2} becomes  
\begin{equation}\label{ODE3D}
	\begin{split}
		\dot{T} &= -cT + UT + V, \\
		\dot{U} &= \frac{1}{\nu}\bigg[-cU + \frac{1}{2}U^2 + \rho V\bigg], \\
		\dot{V} &= T(T-1), \\
	\end{split}
\end{equation}
which is the system that forms the basis for inspection in this section.

This section is broken down into three subsections. The first subsection, \S~\ref{subsec:ViscousExistence}, parallels \S~\ref{subsec:InviscidExistence} by providing the existence of traveling wave solutions of \eqref{FKPP_Burgers} that travel at fast speeds, i.e. $c \gg 1$. This is again achieved by proving the existence of heteroclinic orbits to \eqref{ODE3D} for all $\rho > 0$ that connect $(1,u_0,c-u_0)$ to the origin. Our analysis notably lacks the existence of a minimum wave speed, and therefore the we focus on the existence of traveling wave solutions in the two distinct parameter regimes: $\rho > 0$ small and large. The parameter regime of small $\rho > 0$ is handled in \S~\ref{subsec:ViscousSmallRho} where we prove the points (2) and (3) of Theorem~\ref{thm:Viscous}. In \S~\ref{subsec:ViscousLargeRho} we turn to the large $\rho > 0$ parameter regime to prove Theorem~\ref{thm:Viscous}(4). In total, the results in the following subsections prove the entirety of Theorem~\ref{thm:Viscous} as presented in Section~\ref{sec:Results}.


\subsection{Existence of traveling waves}\label{subsec:ViscousExistence}

The work of this section parallels the work of Subsection~\ref{subsec:InviscidExistence} applied to system~\eqref{ODE3D}. In particular, we show that for any $\nu,\rho > 0$ and sufficiently large $c > 0$, system~\eqref{ODE3D} has a heteroclinic trajectory corresponding to a traveling wave solution of \eqref{FKPP_Burgers}. We are interested in heteroclinic orbits that belong to the cube $\mathcal{C}$, given by
\[
	\mathcal{C} = \{(T,U,V):\ 0 \leq T \leq 1,\ 0 \leq U \leq u_0,\ 0 \leq V \leq c-u_0\}.
\]
We present the following lemma outlining the basic properties of the equilibria that make up the desired heteroclinic trajectories and the dynamics in the cube $\mathcal{C}$.

\begin{prop}\label{prop:Equilibria2} 
	For each $c,\nu,\rho > 0$ we have the following:
	\begin{enumerate}[(1)]
		\item The origin of \eqref{ODE3D} is locally asymptotically stable for all $c > 0$. In particular, when $c \in (0,2)$ the linearization of \eqref{ODE3D} about the origin has a single real and negative eigenvalue along with a pair of complex-conjugate eigenvalues with negative real part, and when $c \geq 2$ all eigenvalues are real and negative.
		\item The equilibrium $(1,u_0,c-u_0)$ of \eqref{ODE3D} is hyperbolic with a one-dimensional unstable manifold whose tangent vector at this equilibrium points to the interior of $\mathcal{C}$. 
		\item The unstable manifold of $(1,u_0,c-u_0)$ is necessarily monotone while it remains in $\mathcal{C}$. That is, $\dot{T} < 0$, $\dot{U} < 0$, and $\dot{V} < 0$ everywhere along the trajectory in $\mathcal{C}$.
		\item The unstable manifold of $(1,u_0,c-u_0)$ can only leave $\mathcal{C}$ by crossing $V = 0$.
		\item If the unstable manifold of $(1,u_0, c-u_0)$ leaves $\mathcal{C}$ by crossing $V=0$ and connects to the origin, it must be that $T(\xi)<0$ for some $\xi$.
	\end{enumerate}
\end{prop}

\begin{proof}
	Point (1) is proved by simply inspecting the eigenvalues of the matrix resulting from linearizing \eqref{ODE3D} about the origin and is therefore omitted. Similarly, points (3) and (4) follow in a nearly identical manner to the analogous points in Proposition~\ref{prop:Equilibria}. Therefore, we only prove points (2) and (5) here, starting with the former.
	
	Linearizing \eqref{ODE3D} about the equilibrium point $(1,u_0,c-u_0)$ results in the matrix 
	\begin{equation}\label{Matrix3D}
		\begin{bmatrix}
			u_0 - c & 1 & 1 \\ 0 & \frac{u_0 - c}{\nu} & \frac{\rho}{\nu} \\ 1 & 0 & 0
		\end{bmatrix}
	\end{equation}
	for which the eigenvalues, $\lambda \in \mathbb{C}$, are roots of the characteristic equation:
	\be
		g(\lambda) := \lambda^3 + \bigg(\frac{(c - u_0)(1 + \nu)}{\nu}\bigg) \lambda^2 + \bigg(\frac{[(c - u_0)^2 - \nu]}{\nu}\bigg)\lambda + \frac{u_0 - c - \rho}{\nu}.  
	\ee
	We first note that $g(0) < 0$ by definition of $u_0$, and since $g$ is a cubic polynomial in $\lambda$ with positive leading coefficient, it follows from the intermediate value theorem that there exists a real positive root of $g$ for all $c,\nu,\rho > 0$. We now show that the other roots of $g$ have negative real part. Once we know that $g$ has only one root with positive real part, the fact that the tangent vector of this one-dimensional unstable manifold points into $\mathcal{C}$ follows from a local analysis of the flow along the edges of $\mathcal{C}$ near the equilibrium $(1,u_0,c-u_0)$.
	
	We begin by showing that the roots of $g$ never cross the imaginary axis through variations of the parameters $c,\nu,\rho > 0$. Note that this implies that the equilibrium $(1,u_0,c-u_0)$ is hyperbolic for all relevant parameter values since none of the eigenvalues of \eqref{Matrix3D} can lie on the imaginary axis. Substituting $\lambda = {\rm i\omega}$ for some $\omega \in \R$, into $g(\lambda) = 0$ and separating out into real and imaginary parts gives that $\omega$ solves
	\be
		\begin{split}
			-\omega^3 + \bigg(\frac{[(c - u_0)^2 - \nu]}{\nu}\bigg)\omega &= 0 \\
			\bigg(\frac{(u_0 - c)(1 + \nu)}{\nu}\bigg) \omega^2 + \bigg(\frac{u_0 - c - \rho}{\nu}\bigg) &= 0. 
		\end{split}
	\ee
	But, since $u_0 < c$, it follows that no real value of $\omega$ can be obtained to solve the second equation for any $c,\nu,\rho > 0$, thus proving the claim.
	
	Now that we know that the number of roots of $g$ with negative and positive real parts do not change by varying parameters, we consider the parameter region $\nu \gg 1$ to show that there are two roots with negative real part. Setting $\varepsilon := \nu^{-1}$ and considering $0 < \varepsilon \ll 1$ gives that $g$ can be written
	\[
		g(\lambda) =  \lambda^3 + (c - u_0) \lambda^2 - \lambda + \varepsilon[(c - u_0)\lambda^2 + (u_0-c)^2\lambda + u_0 - c - \rho]. 
	\]
	For $0 < \varepsilon \ll 1$ the function $g$ becomes a regular perturbations of the cubic polynomial $\lambda^3 + (c - u_0) \lambda^2 - \lambda$, and therefore the roots of $g$, denoted $\lambda_1,\lambda_2,$ and $\lambda_3$, are given by
	\[
		\begin{split}
			\lambda_1 &= \frac{1}{2}(u_0 - c - \sqrt{(u_0 - c)^2 + 4}) + \mathcal{O}(\varepsilon), \\
			\lambda_2 &= 0 + \mathcal{O}(\varepsilon), \\
			\lambda_3 &= \frac{1}{2}(u_0 - c + \sqrt{(u_0 - c)^2 + 4}) + \mathcal{O}(\varepsilon).
		\end{split}
	\] 
	Since $u_0 < c$ we have that $\lambda_1 < 0 < \lambda_3$ for all sufficiently small $\varepsilon > 0$. It remains to show that $\lambda_2 < 0$ for sufficiently small $\varepsilon > 0$. Let us introduce $\lambda_2 = \varepsilon\tilde{\lambda}_2$ so that 
	\[
		0 = g(\lambda_2) = g(\varepsilon\tilde{\lambda}_2) = \varepsilon(-\tilde{\lambda}_2 + u_0 - c - \rho) + \mathcal{O}(\varepsilon^2).
	\]
	Hence, $\lambda_2 = (u_0 - c - \rho)\varepsilon + \mathcal{O}(\varepsilon^2)$ for all sufficiently small $\varepsilon > 0$. Since $u_0 - c - \rho < 0$, it follows that $\lambda_2 < 0$ when $0 < \varepsilon \ll 1$, and from the above arguments we have that $g$ always has two roots with negative real parts and one positive real root. The proves Point (2).
	
	In order to establish (5), we argue by contradiction.  We assume that $T\geq 0$ on $\R$ and that $V$ becomes negative for the first time on some interval $(\underline \xi_V, \overline \xi_V)$.  Since $V$ is negative somewhere and $V(\infty) = 0$, it must increase somewhere on $(\underline \xi_V, \overline \xi_V)$.  Using~\eqref{ODE3D} and our contradictory assumption, it follows that $T$ must be larger than $1$ somewhere on $(\underline \xi_V, \overline \xi_V)$.  By Point (3), $T$ and $U$ are decreasing on $(-\infty, \underline \xi_V]$.  Further, by~\eqref{ODE3D}, $U$ is either negative or decreasing as long as $V < 0$.  It follows that $U<c$ on $(\underline \xi_V, \overline \xi_V)$, which, by~\eqref{ODE3D} again, implies that $T$ is decreasing on $(\underline \xi_V, \overline \xi_V)$.  Thus, $T< 1$ on $(-\infty, \overline \xi_V)$, which is a contradiction.
\end{proof} 

Much like the inviscid equation, we see from Proposition~\ref{prop:Equilibria2} that a heteroclinic orbit (if it exists) from $(1,u_0,c-u_0)$ to the origin necessarily\footnote{Actually, there is a small subtlety here.  It is {\em a priori} possible that a heteroclinic connection approaches the origin along the stable manifold associated to the {\em real} eigenvalue of the linearized equation, which allows it to avoid the rotation induced by the complex eigenvalues when $c< 2$.  This can be ruled out by computing the associated eigenvector, which is $(0,1,0)$.  It is, thus, easy to see that this manifold is completely contained in the set $\{(0,U, 0) : U \in \R\}$ and, hence, cannot connect to $(1,u_0,c-u_0)$.} has $T$ and/or $U$ components that are negative when $c \in(0,2)$, thus giving a lower bound on the existence of nonnegative heteroclinic orbits, i.e. $\underline c_*(\nu,\rho) \geq 2$ for all $\nu,\rho > 0$ in the language of Theorem~\ref{thm:Viscous}.

Despite the system \eqref{ODE3D} having one more dimension than that of \eqref{Planar}, from Proposition~\ref{prop:Equilibria2}(4) the analysis is similar since we need only protect the component $V$ from becoming negative to ensure that a desired heteroclinic orbit exists. We show this now.

\begin{lem}\label{lem:TW_char}
	Suppose that $(T,U,V)$ is a solution of~\eqref{ODE3D} with $(T(-\infty), U(-\infty), V(-\infty)) = (1, u_0, c-u_0)$.  Then $(T(\infty), U(\infty), V(\infty)) = (0,0,0)$ and $T\geq 0$ if and only if $(T,U,V)$ leaves $(1,u_0,c-u_0)$ into $\mathcal{C}$ and $V\geq 0$.
\end{lem}
\begin{proof}
We prove the forward direction first.  Suppose that $(T(\infty), U(\infty), V(\infty)) = (0,0,0)$ and $T \geq 0$. First we suppose that $(T,U,V)$ leaves $(1,u_0,c-u_0)$ out of $\mathcal{C}$.  It follows that $T$, $U$, and $V$ are increasing initially by \Cref{prop:Equilibria2}.(2).  In order to connect to the origin, it must be that either $T$ or $U$ must have a maximum.  Suppose that $T$ has a maximum before $U$.  Then, at the maximum, $T>1$, $\ddot T \leq 0$, $\dot T = 0$ and $\dot U \geq 0$.  Plugging this into~\eqref{TWave2}, we find
\[
	0 \leq - c \dot T - \ddot T + \dot U T + U \dot T
		= T(1-T) < 0,
\]
which is a contradiction.  A similar argument shows that the maximum of $U$ cannot occur before $T$.  We conclude that $(T,U,V)$ cannot leave $(1,u_0, c-u_0)$ out of $\mathcal{C}$.  

Next, we suppose that $V<0$ somewhere.  By \Cref{prop:Equilibria2}.(5), it follows that $T$ becomes negative, which is a contradiction of our assumption $T \geq 0$.  We conclude that $V\geq 0$ everywhere.

We now prove the backward direction.   Suppose that $(T,U,V)$ leaves $(1,u_0,c-u_0)$ into $\mathcal{C}$ and $V\geq 0$ always.  By \Cref{prop:Equilibria2}.(4), since $V\geq 0$, it follows that $(T,U,V)$ remains in $\mathcal{C}$.  Thus, $T\geq 0$.  Further, as $\mathcal{C}$ is compact and $T$, $U$, and $V$ are decreasing, we must have that $(T(\infty),U(\infty), V(\infty)) = (0,0,0)$.  This concludes the proof.
\end{proof} 

We are now able to leverage the above characterization of traveling waves, \Cref{lem:TW_char}, in order to establish the existence of traveling waves at sufficiently large speeds for any $\rho$ and $\nu$.

\begin{lem}\label{lem:Het2} 
	For each $\nu,\rho > 0$, there exists a finite $\overline c(\nu,\rho) \geq 2$, which depends on $\rho$ and $\nu$, such that for all $c \geq \overline c(\nu,\rho)$ system \eqref{ODE3D} has a heteroclinic orbit from $(1,u_0,c-u_0)$ to $(0,0,0)$ which remains in $\mathcal{C}$ for all $\xi \in \R$.
\end{lem}

\begin{proof}
	This proof proceeds in a nearly identical manner to that of Lemma~\ref{lem:Het} by showing that, for $c \gg 1$, we can form a region that prevents $V$ from becoming negative. Begin by fixing $\nu,\rho > 0$ and considering the region
	\be\label{Region1}
		\mathcal{R}_1 = \{(T,U,V)\in\mathcal{C}:\ \frac{c}{2}T \leq V \leq c-u_0,\ T\in[0,1]\}.
	\ee
	We show that this region contains the component of the unstable manifold of $(1,u_0,c-u_0)$ that enters in to $\mathcal{C}$ for all $\xi\in\R$. We first discuss how confining this unstable manifold to the region $\mathcal{R}_1$ proves it must converge to the origin as $\xi \to \infty$. Note that Proposition~\ref{prop:Equilibria2} gives that the only way by which this unstable manifold can leave $\mathcal{C}$ is by crossing $V = 0$. If the unstable manifold lies in $\mathcal{R}_1$ up to a point where we get $V = 0$, we notice that by definition we must also have $T = 0$. Notice that $T = V = 0$ is an invariant manifold of \eqref{ODE3D}, and so the trajectory remains at $T = V = 0$ indefinitely. On this invariant manifold in $\mathcal{C}$, we have that $0 \leq U < c$, and so $U$ is decreasing towards $U = 0$ as $\xi \to \infty$, thus giving that the trajectory converges to the origin as $\xi \to \infty$. Hence, we see that if $V = 0$ we must converge to the origin, while if $V > 0$ for all $\xi \in\R$, the discussion prior to this lemma again proves that we converge to the origin. Thus, showing that if the component of the unstable manifold of $(1,u_0,c-u_0)$ that enters in to $\mathcal{C}$ remains in $\mathcal{R}_1$ for all $\xi\in\R$, then it is guaranteed that it converges to the origin, giving the desired heteroclinic orbit.  
	
	We next check that the equilibrium belongs to $\mathcal{R}_1$ for all sufficiently large $c > 0$. This amounts to checking that $c/2 <c - u_0$. Then, for $c \gg 1$ we have  
	\begin{equation}\label{TrapInCube}
		\begin{split}
			\frac{c}{2} - (c - u_0) &= \frac{c}{2} + \rho - \sqrt{c^2 + \rho^2} \\
			&= -\frac{c}{2} + \rho + \mathcal{O}\left(\frac{\rho^2}{c}\right),
		\end{split}	
	\end{equation} 
	which is strictly negative for any $\rho > 0$ and all $c$ taken sufficiently large, depending on the value of $\rho$. Hence, we find that the curve \eqref{Region1} does indeed lie entirely in the cube $\mathcal{C}$ for all $T\in[0,1]$ for all $c$ sufficiently large. 
	
	We now show that $\mathcal{R}_1$ traps the unstable manifold inside of it. Let us assume that at some $\xi_0 \in\R$ this unstable manifold hits the boundary of $\mathcal{R}_1$. From Proposition~\ref{prop:Equilibria2} this must occur at the lower boundary $V = \frac{c}{2}T$. Then, it must be the case that at $\xi = \xi_0$ we have
	\be
		\frac{\drm}{\drm\xi}\bigg(V - \frac{c}{2}T\bigg) \leq 0
	\ee
	when $V = \frac{c}{2}T$ inside of $\mathcal{C}$. From the arguments above we may assume that $T(\xi_0) > 0$, and using system \eqref{ODE3D} at $\xi = \xi_0$ we have that
	\be\label{CubeTrap}
		\begin{split}
			 0 \geq\frac{\drm}{\drm \xi}\bigg(V - \frac{c}{2} T\bigg) &= T(T-1) -\frac{c}{2}\bigg(-cT + UT + \frac{c}{2} T\bigg) \\
			&= -T\bigg(\frac{c^2}{4} + \frac{c}{2}(U - c) + 1 - T\bigg) \\
			&\geq T\bigg(-\frac{c^2}{4} + \frac{c}{2}(\sqrt{\rho^2 + c^2} - \rho) - 1\bigg)
		\end{split}
	\ee 
	where we have used the fact that $0 \leq U \leq u_0$ and $0 \leq T \leq 1$. From here we may proceed in a similar manner to the arguments of Lemma~\ref{lem:Het}. Indeed, when $\rho/c \ll 1$, we have $\sqrt{c^2 + \rho^2} = c + \mathcal{O}(\rho^2/c)$, and, hence,~\eqref{CubeTrap} becomes
	\be\label{CubeTrap2}
		 0 \geq\frac{\drm}{\drm \xi}\bigg(V - \frac{c}{2} T\bigg)
			\geq T\bigg(\frac{c^2}{4} - \frac{c\rho}{2} +  \mathcal{O}\left(\rho^2\right) - 1\bigg).
	\ee
	Hence, $\frac{\drm}{\drm \xi}\bigg(V - \frac{c}{2} T\bigg) > 0$ for sufficiently large $c$. This is clearly a contradiction, thus showing the unstable manifold cannot hit the boundary of $\mathcal{R}_1$. Hence, the region $\mathcal{R}_1$ contains the component of the unstable manifold of $(1,u_0,c-u_0)$ that enters in to $\mathcal{C}$ for all sufficiently large $c > 0$. This therefore completes the proof.
\end{proof} 


\subsection{Small $\rho$ parameter regime}\label{subsec:ViscousSmallRho}

In this section we present results on the speeds for which traveling waves of the PDE \eqref{FKPP_Burgers} in the viscous parameter regime exhibits nonnegative monotone traveling waves when $\rho > 0$ is small. First, examining the proof of \Cref{lem:Het2} (see, in particular,~\eqref{TrapInCube} and~\eqref{CubeTrap2}), it is clear that $\overline c_*(\nu,\rho) \to 2$ as $\rho \to 0$.  We collect this in the following corollary.

\begin{cor}\label{lem:SmallRho2} 
	As $\rho \to 0$, $\overline c(\nu,\rho) \to 2$.
\end{cor}

Having establishing this fact, which has no restriction on the range of $\nu$, we specialize to the small $\nu$ regime in order to establish the threshold result on $\rho$; that is, we show that, if $\rho$ is sufficiently small, the minimal speed is $2$.

\begin{lem}\label{lem:SmallRho3} 
	There exists $\nu_* \geq 2$ so that for all $\nu \in (0,\nu_*)$ there exists $\rho_\nu > 0$ such that for all $c \geq 2$ system~\eqref{Planar} has a heteroclinic orbit from $(1,u_0)$ to $(0,0)$ which remains in $\mathcal{B}$ for all $\xi \in \R$. As a consequence, $\underline c_* (\nu,\rho) = \overline c_*(\nu, \rho) = 2$ for all $(\nu,\rho) \in (0,\nu_*)\times(0,\rho_\nu)$. 
\end{lem}

\begin{proof}
	We proceed via the construction of a trapping region.  The one in \Cref{lem:Het2} is not sufficiently precise in the $\rho \ll 1$ regime (see~\eqref{CubeTrap2}) and, hence, we aim to show that the region
	\be
		\mathcal{R}_2 = \{(T,U,V)\in\mathcal{C}:\ 0\leq \nu U \leq \sqrt{\rho}T,\ \frac{c}{2}T \leq V \leq c-u_0,\ T\in[0,1]\}
	\ee
	contains the unstable manifold of $(1,u_0,c-u_0)$ in $\mathcal{C}$ for all $\xi\in\R$, for sufficiently small $\rho > 0$ and all $c \geq 2$. Throughout we take $\nu \in (0,2)$ to ensure that $\nu$ is bounded. We note that this region is similar to the region used to prove Lemma~\ref{lem:SmallRho2} but also has an upper bound on $U$. 
	
	We first show that this region is well-defined and that the equilibrium $(1,u_0,c-u_0)$ belongs to $\mathcal{R}_2$, thus giving that the component of the unstable manifold of $(1,u_0,c-u_0)$ that enters in to $\mathcal{C}$ initially belongs to $\mathcal{R}_2$. As in the previous proofs, we need only check at $T = 1$. The proof that $\frac{c}{2} < c-u_0$ for sufficiently small $\rho > 0$ and all $c \geq 2$ follows as in the proof of Lemma~\ref{lem:Het2}, and so the bounds on $V$ are well-defined. We therefore need only check that $\sqrt{\rho} > \nu u_0$ to show that $(1,u_0,c-u_0) \in \mathcal{R}_2$ when $\rho > 0$ is small, uniformly in $c \geq 2$. Then,
	\be
		\begin{split}
		\sqrt{\rho} - \nu u_0 &= \sqrt{\rho} - \nu(c + \rho - \sqrt{c^2 + \rho^2}) 
		= \sqrt{\rho} - \nu\rho + \mathcal{O}\bigg(\frac{\rho^2}{c}\bigg),
		\end{split}
	\ee 
	so that for sufficiently small $\rho > 0$ we have that $\sqrt{\rho} > \nu u_0$, uniformly in $c \geq 2$ and $\nu\in (0,2)$. Hence, $(1,u_0,c-u_0) \in \mathcal{R}_2$ in the desired parameter region. 
	
	Let us now assume that there is some first point $\xi_0 \in \R$ at which the unstable manifold of $(1,u_0,c-u_0)$ intersects a boundary of $\mathcal{R}_2$ in $\mathcal{C}$. From Proposition~\ref{prop:Equilibria2}, we have that all components of this unstable manifold are monotonically decreasing, and so this collision must take place at a lower boundary in $(T,U,V)$. First, if $U = 0$,  Proposition~\ref{prop:Equilibria2}(4) implies that $V = 0$ as well. But then, from the definition of $\mathcal{R}_2$, it follows that $(T,U,V) = (0,0,0)$, and so $\xi_0 = \infty$, which would yield that the unstable manifold has converged to the origin as desired. Hence, we may assume that that either (i) $\nu U = \sqrt{\rho}T$ or (ii) $\frac{c}{2} T = V$ at $\xi = \xi_0 < \infty$. We show that neither of these cases are possible.   
	
	Let us begin by assuming that at $\xi_0$ we have $\nu U = \sqrt{\rho}T$.  Since this is a point of first collision,  $V \geq \frac{c}{2}T$ and  
	\be
		\frac{\drm}{\drm \xi} \bigg[\sqrt{\rho}T - \nu U\bigg] \leq 0.	
	\ee 
	Then, assuming that $\nu U = \sqrt{\rho}T$ and $\rho > 0$ is held sufficiently small, from the vector field \eqref{ODE3D} we have
	\be
		\begin{split}
		0 \geq \frac{\drm}{\drm \xi} \bigg[\sqrt{\rho}T - \nu U\bigg]  
		&= \sqrt{\rho}(U - c)T - \bigg(\frac{1}{2}U^2 - cU\bigg) + \sqrt{\rho}(1 - \sqrt{\rho})V \\
		&= \nu(U - c)U - \bigg(\frac{1}{2}U^2 - cU\bigg) + \sqrt{\rho}(1 - \sqrt{\rho})V \\	
		&= \bigg(\frac{c}{2} - \bigg(\nu - \frac{1}{2}\bigg)(c-U)\bigg)U + \sqrt{\rho}(1 - \sqrt{\rho})V \\
		&\geq \bigg(\frac{c}{2} - \bigg(\nu - \frac{1}{2}\bigg)(c-U) + \frac{c}{2}(1 - \sqrt{\rho})\nu\bigg)U,
		\end{split}
	\ee
	where we have used the fact that $\sqrt{\rho}V \geq \frac{c\sqrt{\rho}}{2}T = \frac{c\nu}{2}U$ in the last inequality. In the case that $\nu \in (0,\frac{1}{2})$, we use the fact that $0 < U < u_0 \leq c$ to see that 
	\be 
		0 \geq \frac{\drm}{\drm \xi} \bigg[\sqrt{\rho}T - \nu U\bigg] > 0,
	\ee
	a contradiction. Therefore, we are left with the case that $\nu \in [\frac{1}{2},2)$. Here we again use the fact that $U < c$ to get
	\be \label{e.c9141}
		\begin{split}
		0 \geq \frac{\drm}{\drm \xi} \bigg[\sqrt{\rho}T - \nu U\bigg]  
		&\geq \bigg(\frac{c}{2} - (\nu - \frac{1}{2})c + \frac{c}{2}(1 - \sqrt{\rho})\nu\bigg)U 
		= \left(c - \nu\left( \frac{\sqrt \rho + c}{2}\right)\right)U.
		\end{split}	
	\ee 
	The term on the right is positive if and only if $\nu < 2c/(\sqrt \rho + c)$.  Hence, if $\nu < 2$, there exists $\rho$ sufficiently small depending only on $\nu$ so that the term on the right in~\eqref{e.c9141} is positive.  This is a contradiction, showing that case (i) is not possible.

	Turning now to case (ii), we assume that $\frac{c}{2}T = V$ while $0 < \nu U \leq \sqrt{\rho}T$ at $\xi_0$. Proceeding as in Lemma~\ref{lem:SmallRho2}, it follows that 
	\be
		\begin{split}
			0 &\geq \frac{\drm}{\drm \xi}\bigg(V - \frac{c}{2} T\bigg) 
			= -T\bigg(\frac{c^2}{4} + \frac{c}{2}(U - c) + 1 - T\bigg) \\
			&\geq-T(U - T) 
			\geq \bigg(1 - \frac{\sqrt{\rho}}{\nu}\bigg)T^2,
		\end{split}
	\ee 
	since we have $c \geq 2$ and $\nu U \leq \sqrt{\rho}T$. Taking $\rho > 0$ sufficiently small then implies that  
	\be
		\frac{\drm}{\drm \xi}\bigg(V - \frac{c}{2} T\bigg) > 0,
	\ee 
	a contradiction. Therefore, case (ii) is also impossible, showing that the unstable manifold of $(1,u_0,c-u_0)$ cannot intersect the boundaries of $\mathcal{R}_2$ in $\mathcal{C}$. This concludes the proof.
\end{proof}


\subsection{Large $\rho$ parameter regime}\label{subsec:ViscousLargeRho}

In this subsection we consider \eqref{ODE3D}, and subsequently \eqref{FKPP_Burgers} with $\nu > 0$, in the large $\rho > 0$ parameter regime. In particular, the first presented lemma (Lemma~\ref{lem:LargeRho3}) shows that for each $\nu > 0$, monotone traveling waves of \eqref{FKPP_Burgers} exist at all speeds $c \gtrsim \rho^{1/2}$. The second lemma of this section, Lemma~\ref{lem:LargeRho4}, shows that monotone traveling waves of \eqref{FKPP_Burgers} cannot exist at speeds $c\lesssim \rho^{1/3}$. In the context of Theorem~\ref{thm:Viscous} these lemmas imply that $\overline c_*(\nu,\rho) = \mathcal{O}(\rho^{1/2})$ and $\underline c_*(\nu,\rho) = \mathcal{O}(\rho^{1/3})$ as $\rho \to \infty$. Hence, the minimum speed of traveling waves to \eqref{FKPP_Burgers} with $\nu > 0$ scales like $\rho^{1/3}$, similar to the $\nu = 0$ case proven in Section~\ref{subsec:InviscidSpeed}.

We briefly comment on the obstruction to obtaining a bound of the order $\rho^{1/3}$ on $\overline c_*(\nu,\rho)$, as we did in the inviscid case.  In both cases, the proof proceeds via the construction of a trapping region.  The key step in establishing that a set is a trapping region is to analyze the phase plane on the boundary of the trapping region.  In the inviscid case, because the system is two-dimensional, any curve defining a boundary yields an explicit relationship between $T$ and $U$.  This corresponds to establishing an inequality with only one degree of freedom.  In the viscous case, which is three-dimensional, there are two degrees of freedom, making the argument significantly more complicated.

\begin{lem}\label{lem:LargeRho3} 
	Fix $\sigma_1 > 2$. There exist $\rho_1 > 0$ such that for all $\rho \geq \rho_1$, system \eqref{ODE3D} has a heteroclinic orbit from $(1,u_0,c-u_0)$ to $(0,0,0)$ that remains in $\mathcal{C}$ for all $\xi \in \R$ and $c \geq \sigma_1\rho^{1/2}$. As a consequence,
	\[
		\limsup_{\rho\to\infty} \frac{\overline c_*(\nu,\rho)}{ \rho^{1/2}} \leq 2.
	\]
\end{lem}

\begin{proof} 
	We proceed as in the proof of Lemma~\ref{lem:Het2} by constructing an appropriate trapping region that contains the component of the unstable manifold of $(1,u_0,c-u_0)$ that enters in to $\mathcal{C}$ for all $\xi\in\R$, thus showing that it converges to the origin. 

Fix $\sigma > 2$ and let $c = \sigma \rho^{1/2}$. Define the region
\be
	\mathcal{R}_3 = \left\{(T,U,V) \in \mathcal{C} : \frac{\sigma^2}{4} T \leq V \leq c - u_0\right\}.
\ee
Throughout we will fix $(T,U,V)$ to be the unstable manifold of $(1,u_0,c-u_0)$ that enters in to $\mathcal{C}$. We first show that $(T,U,V) \in \mathcal{R}$ at $\xi = -\infty$.  To this end, we need only check that $(\sigma^2/4) T(-\infty) \leq V(-\infty)$; that is $\sigma^2/4 < c- u_0$.  Notice that
\be
	c - u_0
		= \sqrt{c^2 + \rho^2} - \rho
		= \frac{c^2}{2\rho} + \mathcal{O}(c^4/\rho^3)
		= \frac{\sigma^2}{2} + \mathcal{O}(\rho^{-1}).
\ee
This last term is larger than $\sigma^2/4$ for $\rho$ sufficiently large, confirming that $(T,U,V)\in \mathcal{R}_3$ at $\xi = -\infty$.

Next, we check that $(T,U,V)$ cannot exit $\mathcal{R}_3$.  Arguing by contradiction, we may take $\xi_0$ to be the first time that $(T,U,V)$ reaches a boundary of $\mathcal{R}_3$.  As in \Cref{lem:SmallRho3}, we have that $T, U, V > 0$ at $\xi = \xi_0$.

As $V$ is clearly decreasing, see~\eqref{ODE3D}, it must be that $V(\xi_0) = (\sigma^2/4) T(\xi_0)$.  Since $V(\xi) > (\sigma^2/4) T(\xi)$ for all $\xi < \xi_0$, we conclude that, at $\xi_0$,
\be
	0 \geq \dot V - \frac{\sigma^2}{4} \dot T.
\ee
Using~\eqref{ODE3D} and recalling that $V = (\sigma^2/4) T$, we find
\be
	0
		\geq - T(1-T) - \frac{\sigma^2}{4} \left( - (c-U) T +V\right)
		= T \left(T - 1 - \left(\frac{\sigma^2}{4}\right)^2 + \frac{\sigma^2}{4}(c-U)\right).
\ee
Since $U$ is decreasing by \Cref{prop:Equilibria2}, it follows that
\be
	c- U \geq c-u_0 = \frac{\sigma^2}{2} +\mathcal{O}(\rho^{-1}).
\ee
Hence, we obtain
\be
	0 \geq T \left(T - 1 - \frac{\sigma^4}{16} + \frac{\sigma^2}{4}\frac{\sigma^2}{2} + \mathcal{O}(\rho^{-1})\right)
	= T \left(T - 1 +\frac{\sigma^4}{16} + \mathcal{O}(\rho^{-1})\right).
\ee
Recalling that $\sigma > 2$ and using the fact that $T \in [0,1]$, it follows that the rightmost expression above is strictly positive for sufficiently large $\rho > 0$. This is a contradiction, thus showing that $\mathcal{R}_3$ forms a trapping region for \eqref{ODE3D}, as desired.
\end{proof} 

We now obtain a lower bound on $\underline c_*(\nu,\rho)$ that exhibits the $\rho^{1/3}$ scaling, thus completing the proof of Theorem~\ref{thm:Viscous}.

\begin{lem}\label{lem:LargeRho4} 
	The system~\eqref{ODE3D} does not have a heteroclinic orbit from $(1,u_0, c-u_0)$ to $(0,0,0)$ that remains in $\mathcal{C}$ for all $\xi\in \R$ for any $c < \rho^{1/3}$.  As a consequence,
	\[
		\underline c_*(\nu,\rho)
			\geq \max\{2,\rho^{1/3}\}.
	\]
\end{lem}
\begin{proof}
	The proof begins similarly as \Cref{lem:LargeRho3} and uses the same convention that all integrals are over $\xi \in (-\infty,\infty)$.  Indeed, as in \Cref{lem:LargeRho3}, it is possible to establish
	\be\label{e.c832}
		\frac{1}{2}
			\leq \int (c-U) T(1-T) d\xi
	\ee
	in exactly the same manner.  Unfortunately, the exact computations for the right hand side above that were used in \Cref{lem:LargeRho3} do not work here as the form of $\dot U$ is significantly more complicated.
	
	Taking $\xi \to -\infty$ in the equation for $T$ in~\eqref{ODE3D}, we find
	\[
		\int T(1-T) d\xi
			= c - u_0.
	\]
	Combining this and~\eqref{e.c832} and recalling that $c- U < c$, yields
	\be\label{e.c833}
		\frac{1}{2c}
			\leq \frac{1}{c} \int (c-U) T(1-T) d\xi
			\leq \int T(1-T) d\xi
			= c- u_0.
	\ee
	If $c \geq \rho^{1/3}$, we are finished.  Hence, we may assume that $c < \rho^{1/3}$, in which case we can use a Taylor expansion to find
	\[
		c-u_0
			= \sqrt{c^2 + \rho^2} - \rho
			= \frac{c^2}{2\rho} + \bigO\left(\rho^{-5/3}\right).
	\]
	Plugging this into~\eqref{e.c833} yields
	\[
		\frac{1}{2c}
			\leq \frac{c^2}{2\rho} + \bigO\left( \rho^{-5/3}\right).
	\]
	Rearranging this and recalling that $c< \rho^{1/3}$, we find
	\[
		\rho + \bigO\left(\rho^{-1/3}\right)
			\leq c^3,
	\]
	which concludes the proof.
\end{proof}


\section{Discussion}\label{sec:Discussion} 

In this work we have proven the existence of monotone traveling wave solutions to an FKPP-Burgers equation. We saw that the analysis of the system breaks into two distinct parts, termed the inviscid and the viscous cases, which follow from the analogous nomenclature in the Burgers equation. We began by proving that in the case of the inviscid equation there exists a minimum wave speed so that traveling wave solutions exist for all speeds above the minimum speed and none below. We further showed that these waves are monotone and satisfy the ordering \eqref{Ordering}. It was shown that there exists a critical value of $\rho > 0$ for which below this value the minimum wave speed is exactly 2, as in the FKPP equation (i.e. $u \equiv 0$), while for large values of $\rho > 0$ the speed is strictly larger than 2 and behaves like $\rho^{1/3}$ as $\rho \to \infty$. This transition from the linearly determined wave speed 2 to one that is strictly larger than 2 represents the wave moving from being `pulled' when $\rho$ is small to `pushed' (by the advection $U$) when $\rho$ is large.  

Our work in the viscous case showed that many the properties of traveling wave solutions to the inviscid equation carry over. Notably lacking in our proofs is the existence of a minimum wave speed for all relevant parameter values, but it was shown that when $\nu,\rho > 0$ are taken small a minimum wave speed exists and is exactly 2. This again shows that at least when $\nu > 0$ is small, there is a transition from pushed to pulled waves as $\rho$ is increased. Despite not having a minimum wave speed, we were able to provide bounds for all $\nu,\rho > 0$ for which monotone traveling waves exist at all speeds above one value and no such waves exist below another. Importantly, these bounds scale at least like $\rho^{1/3}$ and at most like $\rho^{1/2}$ as $\rho\to\infty$ for all $\nu >0$, thus showing the `pushed' phenomenon as in the inviscid case. Our theoretical analysis in both the inviscid and viscous cases was complemented with numerical simulations in the small $\rho$ parameter regime and novel computational bounding techniques helped us to determine the pre-factor of the asymptotic expansion in the large $\rho$ regime. Our numerical findings display significant agreement with our analysis and can be used to motivate further investigations on the FKPP-Burgers system.  

There are a number of open questions that can be addressed in a follow-up investigation. First and foremost is proving the existence of a minimum wave speed for the viscous case. Although $\overline c_*(\nu,\rho)$ and $\underline c_*(\nu,\rho)$ coincide when $\nu > 0$ is sufficiently small, it is important to determine if they are indeed equal for all $\nu,\rho > 0$. Beyond this, one would conjecture that the minimum wave speeds (if they exist) are continuous in all parameters, with the numerical evidence presented in Section~\ref{sec:Results} indicating that they are monotone in $\nu,\rho > 0$ as well. Significantly more difficult would be to determine the form of the branching behaviour of the minimum wave speed at $\rho_\nu > 0$ where the waves go from pulled to pushed. In addition, an interesting question is to determine the exact form of
\be
	\lim_{\rho\to\infty} \frac{c_*(\rho)}{\rho^{1/3}}.
\ee
We note that it is currently unknown if such a limit exists, however, it is expected that it does.  When $\nu = 0$, we expect $(3/2)^{1/3}$ while the behavior is unknown when $\nu > 0$ (recall that we cannot yet prove that $\overline c_*(\nu,\rho) = \mathcal{O}(\rho^{1/3})$ when $\nu>0$). Finally, we saw that the proofs of our main results relied heavily on the specific form of the nonlinearities on the right-hand-side of \eqref{FKPP_Burgers}, and therefore it becomes interesting to see what, if any, effect introducing different nonlinearities produces.


\appendix


\section{Bounding the Minimum Wave Speeds Numerically}

Here we briefly discuss the numerical bounding algorithm that is used to produce the values in Table~\ref{table1}. In complete generality the method relies on obtaining auxiliary functions defined by convex differential inequalities whose existence provides trapping boundaries defined by the level sets of the auxiliary function. To be explicit, suppose we are given an ODE
\be\label{appODE}
	\dot{y} = F(y)
\ee   
with $F:\R^n \to \R^n$ and some subset of phase-space $\Omega \subseteq \R^n$. We aim to find a continuously differentiable function $H:\Omega \to \R$ whose zero level set forms a trapping boundary of \eqref{appODE} in $\Omega$. Following \cite{Bramburger}, one way to obtain such a function is to determine the existence of a constant $\lambda > 0$ such that 
\be\label{AuxFn}
	\lambda F(y)\cdot \nabla H(y) \leq -H(y), \quad \forall y\in\Omega.
\ee
Indeed, the term $F(y)\cdot \nabla H(y)$ is exactly the derivative of $H$ along trajectories of \eqref{appODE}, and so we see that the region $\{y\in\Omega:\ H(y) \leq 0\}$ is forward invariant with respect to the dynamics of \eqref{appODE}. Crucially, searching for $H$ over some convex class of functions gives that the set of functions satisfying \eqref{AuxFn} is convex as well.   

In the above analysis we have used the fact that traveling waves of \eqref{FKPP_Burgers} correspond to heteroclinic orbits of a related spatial ODE. To use auxiliary functions to either confirm existence or non-existence of heteroclinic orbits requires coupling the existence of a function $H$ satisfying \eqref{AuxFn} with other constraints that are specific to the ODE in question. These conditions are detailed explicitly below as they apply to the ODEs \eqref{Planar} and \eqref{ODE3D}. In general, providing existence of a heteroclinic orbit using \eqref{AuxFn} would require that the target and source equilibria of the desired heteroclinic orbit both satisfy $H(y) \leq 0$, while the boundaries of $\Omega$ for which the unstable manifold of the source equilibrium could escape $\Omega$ have $V(y) > 0$. Proving non-existence of a heteroclinic orbit is slightly simpler since we could have the source equilibrium lying in the interior of the forward invariant region defined by $H(y) \leq 0$, while the target lies in the complement of this set.  

In the present investigation the wave speed $c > 0$ functions as an external parameter in the ODEs in which we wish to confirm the existence or non-existence of heteroclinic orbits. Given the existence of a minimum wave speed, one approach would be to find the extremal value of $c$ for which an auxiliary function $H$ can be obtained to confirm the existence or non-existence a desired connection in phase-space. Unfortunately, this leads to a non-convex optimization problem (see \cite[Section~2c]{Bramburger}) and so, to implement the process numerically, we perform computations at multiple fixed values of $c$. For example, to find the smallest value of $c$ for which a heteroclinic connection exists, we perform the following iterative procedure. We begin with a sufficiently large value of $c$ for which a heteroclinic orbit exists and another sufficiently small value for which it does not. We can then repeatedly bisect in $c$, attempting to find a (potentially) different auxiliary function $H$ that confirms the existence of a heteroclinic orbit in $\Omega$ at each new value of $c$. The smallest such $c$ for which this can be performed then becomes an upper bound on the minimum wave speed. Confirming non-existence is similarly performed through such a bisection method in the wave speed $c$. All other parameters in the system (i.e. $\nu,\rho > 0$) are held fixed.       

We refer the reader to~\cite{Bramburger} for a more complete discussion of the above method of bounding the wave speed and how to implement this procedure numerically. Briefly, after determining the inequalities that a desired auxiliary function $H$ must satisfy to confirm existence or non-existence of a heteroclinic orbit, we search for $H$ numerically as a degree $d \geq 1$ polynomial in $y$ with tunable real-valued coefficients. These convex inequalities for $H$ in the space $\R[y]_d$, the set of all polynomials with real-valued coefficients and degree $\leq d$, defines a semidefinite program that can be relaxed to a series of sum-of-squares constraints that are numerically tractable. The numerical implementation may either find admissible values for the coefficients of $H$ or return that no such values exist. The results in Table~\ref{table1} were obtained using the MATLAB software YALMIP (version R20190425) \cite{Lofberg2004, Lofberg2009} to translate the sum-of-squares constraints into semidefinite programs which are then solved using Mosek (version 9.0) \cite{Mosek}. The code to reproduce the values in Table~\ref{table1} is available at the repository {\bf GitHub/jbramburger/FKPP-Burgers}. Computations are performed by optimizing over $\lambda > 0$ at each fixed degree $d \geq 1$ of $H$. The degree is successively increased until convergence in the bound is observed, typically around $d = 10$ to $d = 14$.

We conclude this section with a brief discussion of the specific conditions put on $H$ to bound the wave speed from above and below in the cases $\nu = 0$ and $\nu \neq 0$, respectively.

\underline{\bf Upper bounds on $c_*(\rho)$:} From the work of Section~\ref{sec:Inviscid}, our ODE of interest is the planar system \eqref{Planar} and the region of interest is $\Omega = \mathcal{B}$, defined in \eqref{BoxDef}, for each fixed $\rho > 0$ and $\nu = 0$. From Proposition~\ref{prop:Equilibria} we have that the unstable manifold of $(1,u_0)$ can only leave $\mathcal{B}$ by crossing $U = 0$, and so we impose the conditions
\be
	\begin{split}
		-\lambda\bigg[(c-U)\bigg(-cT + UT + \frac{U}{2\rho}(2c - U)\bigg)H_T(T,U) + \rho T(T-1)H_U(T,U)\bigg] &\geq (c- U)H(T,U), \\
		H(T,0) &\geq \varepsilon T(1-T), \\
		-\varepsilon &\geq H(1,u_0) \\
		H(0,0) &= 0,
	\end{split}
\ee 
for all $(T,U)\in\mathcal{B}$, where subscripts denote partial differentiation. Indeed, the first condition is \eqref{AuxFn} for the specific ODE \eqref{Planar}, multiplied through by the positive term $c-U$ to maintain that the inequality is entirely stated in terms of polynomial functions when $V$ is polynomial. For any $\varepsilon > 0$, the second condition guarantees that the set $\{(T,0):\ 0 < T < 1\}\subset\mathcal{B}$ is such $H(T,0) > 0$ for all $0 < T < 1$, thus confirming that the unstable manifold of $(1,u_0)$ cannot cross $U=0$. The results in Table~\ref{table1} always use $\varepsilon = 10^{-4}$. We comment that the presence of $\varepsilon > 0$ comes from the fact that strict inequalities cannot be handled numerically. The third condition gives that the equilibrium $(1,u_0)$ lies in the interior of the forward invariant region, thus confirming that its unstable manifold in $\mathcal{B}$ also lies in this set. Finally, the fourth condition gives that the origin lies on the boundary of the forward invariant region. Finally, we note that we cannot have $H(0,0) \leq -\varepsilon$ since this would be inconsistent with the second condition, thus leading to the requirement that $H(0,0) = 0$, as stated above.    
\\

\underline{\bf Lower bounds on $c_*(\rho)$:} Again we take $\Omega = \mathcal{B}$ and consider the ODE \eqref{Planar}. The conditions to confirm non-existence of a heteroclinic orbit for each fixed $\rho > 0$ are then
\be
	\begin{split}
		-\lambda\bigg[(c-U)\bigg(-cT + UT + \frac{U}{2\rho}(2c - U)\bigg)H_T(T,U) + \rho T(T-1)H_U(T,U)\bigg] &\geq -(c- U)H(T,U), \\
		-H(T,0) &\geq \varepsilon U(1-U), \\
		H(1,u_0) &= 0, \\
		H(0,0) &\geq \varepsilon,
	\end{split}
\ee 
for all $(T,U)\in\mathcal{B}$. The first condition is a variant of \eqref{AuxFn}, in this case of the form
\be\label{AuxFn2}
	\lambda F(y)\cdot \nabla H(y) \leq H(y),	
\ee
which again gives that $H \leq 0$ is forward invariant with respect to the dynamics of $F$. The reason for this change in conditions is to maintain that the third and fourth conditions, which separate the target and source equilibria, are consistent with the condition guaranteeing forward invariance of $H \leq 0$. Indeed, using \eqref{AuxFn} in the place of the first condition above would require that $H(0,0) = 0$, thus not necessarily separating the equilibria and in turn not necessarily prove non-existence of a heteroclinic orbit since both equilibria would belong to the closed set $H \leq 0$. The second condition gives that the line $\{(1,U\: 0 < U < 1\}\subset\mathcal{B}$ lies in the forward invariant region, and since $(1,u_0)$ is a saddle, a local analysis near this equilibrium reveals that its unstable manifold in $\mathcal{B}$ must enter into the forward invariant set $H \leq 0$. As in the upper bounds on $c_0(\rho)$, the second condition guarantees that $H(1,u_0)$ cannot be taken to be negative, and therefore we can only have $H(1,u_0) = 0$, in turn necessitating the alternative first condition.  
\\

\underline{\bf Upper bounds on $\overline c_*(\nu,\rho)$:} Here now we consider the ODE \eqref{ODE3D} with $\Omega = \mathcal{C}$ for $\nu,\rho > 0$ fixed. From Proposition~\ref{prop:Equilibria2} we have that the unstable manifold of $(1,u_0,c-u_0)$ that enters into $\mathcal{C}$ can only leave by crossing $V = 0$. We then seek $H:\R^3\to\R$ satisfying 
\be
	\begin{split}
		-\lambda\bigg[(-cT + UT + V)H_T(T,U,V) + &\frac{1}{\nu}\bigg(-cU + \frac{1}{2}U^2 + \rho V\bigg)H_U(T,U,V) + T(T-1)H_V(T,U,V)\bigg] \geq H(T,U,V), \\
		H(T,U,0) &\geq \varepsilon T(1-T) + \varepsilon U(u_0 - U), \\
		-\varepsilon &\geq H(1,u_0,c-u_0), \\
		H(0,0,0) &= 0,
	\end{split}
\ee 
for all $(T,U,V)\in\mathcal{C}$. The first condition is exactly \eqref{AuxFn} with $F$ specified by the right-hand-side of \eqref{ODE3D}. The remaining conditions are analogous to those for the upper bounds on $c_0(\rho)$.
\\

\underline{\bf Lower bounds on $\underline c_*(\nu,\rho)$:} Again we consider the dynamics \eqref{ODE3D} with $\Omega = \mathcal{C}$ for fixed $\nu,\rho > 0$. To determine non-existence of a heteroclinic connection from $(1,u_0,c-u_0)$ that remains in $\mathcal{C}$ we seek a function $H:\R^3\to\R$ satisfying
\be
	\begin{split}
		-\lambda\bigg[(-cT + UT + V)H_T(T,U,V) + &\frac{1}{\nu}\bigg(-cU + \frac{1}{2}U^2 + \rho V\bigg)H_U(T,U,V) + T(T-1)H_V(T,U,V)\bigg] \geq H(T,U,V), \\
		-\varepsilon &\geq H(1,u_0,c-u_0), \\
		H(0,0,0) &= 0,
	\end{split}
\ee 
for all $(T,U,V)\in\mathcal{C}$, along with the condition
\be\label{AuxFn3}
	H(T,U,V) \geq 0, \quad \forall T,U,V\in [0,\delta].
\ee 
The first condition above represents \eqref{AuxFn}, while the following two conditions put the source equilibrium $(1,u_0,c-u_0)$ in the interior of the forward invariant region and the origin on the boundary. For some sufficiently small $\delta > 0$, the condition \eqref{AuxFn3} works to guarantee that a region around the origin lies outside of the forward invariant region $H \leq 0$ when $H$ is a polynomial in $(T,U,V)$. In our numerical implementations we take $\delta = 0.05$; larger values of $\delta$ lead to less precise bounds, while smaller values show little change in the lower bound.


\bibliographystyle{siam}
\bibliography{FKPP_Burgers}

\end{document}